\newcommand{\Aut}{\mathop{\mathrm{Aut}}\nolimits}
\newcommand{\End}{\mathop{\mathrm{End}}\nolimits}
\newcommand{\Ext}{\mathop{\mathrm{Ext}}\nolimits}
\newcommand{\GL}{\mathop{\mathrm{GL}}\nolimits}
\newcommand{\Hom}{\mathop{\mathrm{Hom}}\nolimits}
\newcommand{\indlim}{\mathop{\underrightarrow{\lim}}\limits}
\newcommand{\Maps}{\mathop{\mathrm{Maps}}\nolimits}
\newcommand{\prlim}{\mathop{\underleftarrow{\lim}}\limits}
\newcommand{\Sy}{\mathop{\mathfrak{S}}\nolimits}
\documentclass[11pt,twoside]{amsart} 
\usepackage{amssymb}
\usepackage{amsmath}
\advance \topmargin by -\headheight
\advance \topmargin by -\headsep
\evensidemargin \oddsidemargin
\date{}
\newtheorem{theorem}{Theorem}[section]
\newtheorem{conjecture}[theorem]{Conjecture}
\newtheorem{corollary}[theorem]{Corollary}
\newtheorem{defin}[theorem]{Definition}
\newtheorem{lemma}[theorem]{Lemma}
\newtheorem{proposition}[theorem]{Proposition}
\title{On semilinear representations of the infinite symmetric group} 
\author{M.Rovinsky} 
\address{National Research University Higher School of Economics, 
AG Laboratory HSE, 7 Vavilova str., Moscow, Russia, 117312
\& Institute for Information Transmission Problems of Russian 
Academy of Sciences}
\email{marat@mccme.ru}
\begin{document} 
\begin{abstract} In this note the smooth (i.e. with open stabilizers) 
linear and {\sl semilinear} representations of certain permutation 
groups (such as infinite symmetric group or automorphism group of an 
infinite-dimensional vector space over a finite field) are studied. 
Many results here are well-known to the experts, at least in the case 
of {\sl linear representations} of symmetric group. The presented results 
suggest, in particular, that an analogue of Hilbert's Theorem 90 should 
hold: in the case of faithful action of the group on the base field the 
irreducible smooth semilinear representations are one-dimensional (and 
trivial in appropriate sense). \end{abstract}
\maketitle 

Throught the paper, $G$ denotes a {\sl permutation group}. 
\begin{defin} A {\bf permutation group} is a Hausdorff topological group 
$G$ admitting a base of open subsets consisting of the left and right 
shifts of subgroups. \end{defin} 
If we denote by $B$ a collection of open subgroups such that the finite 
intersections of conjugates of elements of $B$ form a base of open 
neighbourhoods of $1$ in $G$ (e.g., the set of all open subgroups of $G$), 
then $G$ acts faithfully on the set $\Psi:=\coprod_{U\in B}G/U$, so (i) 
$G$ becomes a {\sl permutation group of} $\Psi$, (ii) the shifts of the 
pointwise stabilizers $G_T$ of the finite subsets $T\subset\Psi$ form 
a base of the topology of $G$. 
Clearly, $G$ is totally disconnected. 

In the most general setting, we are interested in the {\sl continuous} 
$G$-actions on the {\sl discrete} sets (i.e. with open stabilizers; 
they are called {\sl smooth} in what follows). In practice, the considered 
$G$-sets are endowed with extra structures, e.g., of a vector $k$-space 
for a field $k$. In that case our primary goal is a description of the 
{\sl smooth} representations of $G$ in the $k$-vector spaces, especially 
of the irreducible ones. The smooth representations form a $k$-linear 
Grothendieck category with $\bigoplus_{s\in S}W^{\otimes^s_k}$ as one of 
many possible generators, where $W=k[\Psi]$ and $S\subseteq\mathbb N$ is 
an infinite set. However, structure of the $k[G]$-modules $W^{\otimes^s_k}$ 
can be quite complicated. 

Here and throughout the paper, for any abelian group $P$ and a set $S$, 
we denote by $P[S]$ the abelian group consisting of the finite formal 
sums $\sum_ia_i[s_i]$ for all $a_i\in P$, $s_i\in S$. If $P$ is a left 
module over a ring $A$ then $P[S]$ is naturally a left $A$-module. 

For any group $G$ and any field $k$ there is a field extension $K|k$ 
endowed with a faithful $k$-linear $G$-action. Namely, as $K$ one can 
take the fraction field of the symmetric algebra of a faithful 
representation $W$ of $G$ over $k$. Then there is a natural surjection 
$W\otimes_kK\to K$ of $K$-{\sl semilinear} representations of $G$. 
(The semilinear representations are defined in \S\ref{semilin-repr}.) 

This leads to the problem of describing smooth semilinear representations 
over $K$ for a field $K$ endowed with a smooth faithful $G$-action. 
Another motivation for the study of the semilinear representations comes 
from the guess that there are quite few smooth irreducible semilinear 
representations. Two examples of this phenomenon are given by Theorems 
\ref{Satz90} and \ref{smooth-simple}. 
\begin{theorem}[Hilbert's Theorem 90] \label{Satz90} 
Suppose that $G$ is precompact, i.e., any open subgroup of $G$ is of 
finite index. Then any smooth semilinear representation $V$ of $G$ over 
$K$ is trivial: the natural map $V^G\otimes_{K^G}K\to V$ is an isomorphism, 
so $V\cong\bigoplus_IK$ for a basis $I$ of the $K^G$-vector space $V^G$. \end{theorem}
In the case of finite $G$ Theorem \ref{Satz90} is shown in \cite[Satz 1]{Speiser}. 
It implies that 
(i) the functor $H^0(G,-):V\mapsto V^G$ is an equivalence between the category of smooth semilinear representations 
of $G$ over $K$ and the category of vector spaces over the fixed field $K^G$, 
(ii) the symmetric algebra of any smooth faithful $k$-representation of $G$ 
is a generator of the category of smooth $k$-representations of $G$, (iii) 
any cyclic smooth $k[G]$-module can be embedded into an arbitrary field 
extension $K$ of $k$ endowed with a smooth faithful $G$-action trivial 
on $k$. 

\vspace{4mm}

For any set $\Psi$, denote by $\Sy_{\Psi}$ the group of {\sl all} 
permutations of $\Psi$. \begin{theorem}[\cite{H90}] \label{smooth-simple} 
Let $K=k(\Psi)$ be the field of rational functions over a field $k$ in 
variables from a set $\Psi$, endowed with the natural $\Sy_{\Psi}$-action 
by permuting the variables. Then any smooth $K$-semilinear representation 
of $\Sy_{\Psi}$ of finite length is isomorphic to a direct sum of copies 
of $K$. \end{theorem} 
Recall that {\sl length} of an object of a concrete category is 
defined as the maximal number of its proper subobjects in the chains 
of its subobjects. E.g., length is cardinality for the category of sets; 
length is dimension for the categories of projective or vector spaces.

There is a (rather wild) description of all smooth semilinear 
$G$-actions on a given $K$-vector space in the case when a dense 
subgroup of $G$ is exhausted by precompact subgroups, cf. Appendix 
\ref{exhaust}. This description is quite explicit in the case 
when $G$ is exhausted by {\sl open} precompact subgroups. 

Our principal examples of $G$ will be permutation groups of 
{\sl $\Sy$-type}, cf. Definition \ref{def-S-type}. Typical groups 
of $\Sy$-type are the group $\Sy_{\Psi}$, or the automorphism group of 
an infinite-dimensional projective space $\Psi$ over a finite field 
${\mathbb F}_q$, or the automorphism group $\GL_{{\mathbb F}_q}(\Psi)$ of an 
infinite-dimensional vector space $\Psi$ over a finite field ${\mathbb F}_q$. 

Such groups admit a dense subgroup exhausted by finite subgroups. 

The standard generalization of the finite-dimensional representations 
of a permutation group $G$ is the notion of an admissible 
representation: a representation $V$ of $G$ is called {\sl admissible} 
if $\dim V^U<\infty$ for each open subgroup $U\subseteq G$. 

It is well-known (cf., e.g., \cite[\S6 and references therein]{SaSn}) 
that over a field of characteristic zero (i) any smooth cyclic representation of 
$\Sy_{\Psi}$ is admissible and of finite length (cf., e.g., Lemma~\ref{lin-fin-length}), 
(ii) the isomorphism classes of irreducible smooth representations of $\Sy_{\Psi}$ are 
in one-to-one correspondence with (finite) Young diagrams. However, the approach of 
loc.cite is different and I was unable to find in the literature a description of 
{\sl all} injective smooth representation of $\Sy_{\Psi}$. 

One of the purposes of this note is to examine the similar 
questions in the context of semilinear representations. 

The results are as follows: 
\begin{enumerate} \item a description of injective smooth representations of the automorphism group 
of a countable $\Psi$, which is either a set or a vector space over a finite field
(Theorem~\ref{injective-cogenerators}); 
\item a smooth linear representation of a group of $\Sy$-type is of finite 
length if and only if it is finitely generated (Lemma~\ref{lin-fin-length}); 
\item for any smooth $\Sy_{\Psi}$-action on a field $K$, 
the category of smooth $K$-semilinear representations of the group $\Sy_{\Psi}$ 
is locally noetherian, i.e., any smooth 
finitely generated $K$-semilinear representation of the group $\Sy_{\Psi}$ 
is noetherian, cf. Corollary \ref{noetherian}; 
\item a generalization of the well-known cyclicity of finite-dimensional 
semilinear representations of infinite semigroups to the cyclicity of 
certain smooth finitely generated $K$-semilinear representation of $G$, 
for a large class of smooth $G$-actions on fields $K$ (Lemma \ref{surjectivity}); 
\item an example $(K\langle\Sy_{\Psi}\rangle,K[\Psi],K)$ of a triple 
$(A,M,P)$ consisting of an associative unital ring $A$, a noetherian 
$A$-module $M$ and a simple $A$-module $P$ such that (i) any quotient of $M$ 
by any non-zero submodule is isomorphic to a finite direct sum of copies 
of $P$, (ii) for any integer $N\ge 0$ there is a quotient of $M$ isomorphic 
to a direct sum of $N$ copies of $P$ (Lemma \ref{weight-1-quotients} and 
Lemma \ref{triviality-finite-dim}; notations are in \S\ref{semilin-repr}). 
\end{enumerate} 

If $K$ is the field of rational functions over a field $k$ in variables 
enumerated by a set $\Psi$ and $G$ is the symmetric group of $\Psi$, 
there are some reasons to expect an explicit description of the 
indecomposable injectives of the category of smooth $K$-semilinear 
representations of $G$, cf. Conjecture \ref{indec-injectives}. 
This is compatible with Theorem \ref{smooth-simple}. 

\subsection{Permutation groups and categories associated to collections of their 
subgroups} It is well-known (e.g., \cite[Expos\'{e} IV, \S2.4--2.5]{sga-4-1} 
or \cite[\S8.1, Example 8.15 (iii)]{topos}) that the category Sm-$G$ of 
smooth $G$-sets and their $G$-equivariant maps is a topos. For a base 
$B$ of open subgroups of $G$, considered as a poset, let $\mathcal C_B$ 
be the small full subcategory of Sm-$G$ whose objects are the images 
of the contravariant functor $B\to\text{Sm-}G$, $U\mapsto[U]:=G/U$. 
Thus, any morphism in $\mathcal C_B$ is epimorphic and 
$\mathcal C_B([U],[V]):=\Maps_G(G/U,G/V)=(G/V)^U=
\{g\in G~|~gVg^{-1}\supseteq U\}/V$. 

We endow $\mathcal C_B$ with the maximal topology, i.e. we assume that 
any sieve is covering. Then the sheaves of sets, groups, etc. on 
$\mathcal C_B$ are identified with the smooth $G$-sets, groups, etc.: 
${\mathcal F}\mapsto\indlim_{U\in B}{\mathcal F}(U)$ (this is a smooth 
$G$-set, since any element in it belongs to the image of some 
${\mathcal F}(U)$ and the $U$-action on it is trivial by definition) 
and $W\mapsto[U\mapsto W^U]$. 

\section{Skew group rings and semilinear representations}
\label{semilin-repr} 
We use the following slightly more general setting. 

Let $A$ be a (unital) associative ring, $G$ be a semigroup acting on $A$, 
i.e., a (unital) semigroup homomorphism $\rho:G\to\End_{\mathrm{ring}}(A)$ 
is given. Denote by $A\langle G\rangle_\rho=A\langle G\rangle$ the unital 
associative subring in $\End_{{\mathbb Z}}(A[G])$ generated by the natural 
left action of $A$ and the diagonal left action of $G$ on $A[G]$. In other 
words, $A\langle G\rangle$ is the ring of $A$-valued measures on $G$ with 
finite support. Then $A\langle G\rangle$ is a $k$-algebra, where 
$k:=A^{\rho(G)}$ is the fixed ring. If $\rho$ is injective then 
$A\langle G\rangle$ is a {\sl central} $k$-algebra. 

More explicitly, as a left $A$-module $A\langle G\rangle$ coincides 
with $A[G]$. Multiplication is a unique distributive one such that 
$(a[g])(b[h])=ab^{\rho(g)}[gh]$, where we write $a^h$ for the 
result of applying of $h\in\End_{\mathrm{ring}}(A)$ to $a\in A$. 

An additive action of $G$ on an $A$-module $V$ is called 
{\sl semilinear} if $g(a\cdot v)=a^g\cdot gv$ for any $g\in G$, $v\in V$ 
and $a\in A$. Then an $A$-module endowed with an additive semilinear 
$G$-action is the same as an $A\langle G\rangle$-module. 

The principal example of $A$ will be a field. 

For a field $K$ endowed with a $G$-action, a $K$-{\sl semilinear 
representation} of $G$ is a left $K\langle G\rangle$-module. 
We say that a $K$-semilinear representation of $G$ is 
{\sl non-degenerate} if the action of each element of $G$ is injective. 
(Of course, this condition is redundunt if $G$ is a group rather than 
a semigroup.) We omit the $G$-action on $K$ from notation and denote 
by $k:=K^G$ the fixed field. 

The non-degenerate $K$-semilinear representations of $G$ form an abelian 
tensor $k$-linear category. 
The category of smooth $K$-semilinear representations of $G$ is also $k$-linear abelian. 

\vspace{4mm}
The following result will be used in the particular case of the trivial 
$G$-action on the $A$-module $V$ (i.e., $\chi\equiv id_V$), claiming the 
injectivity of the natural map $A\otimes_{A^G}V^G\to V$ (since $V_{id_V}=V^G$). 
\begin{lemma} \label{inject} Let $A$ be a 
division ring endowed with a 
$G$-action $G\to\Aut_{\mathrm{ring}}(A)$, $V$ be a $A\langle G\rangle$-module 
and $\chi:G\to\Aut_A(V)$ be an invertible $G$-action on the $A$-module $V$. 

Set $V_{\chi}:=\{w\in V~|~\sigma w=\chi(\sigma)w
\text{ {\rm for all} $\sigma\in G$}\}$. 

Then $V_{\chi}$ is a $A^G$-module and the natural map 
$A\otimes_{A^G}V_{\chi}\to V$ is injective. \end{lemma} 
\begin{proof} This is well-known: Suppose that some elements 
$w_1,\dots,w_m\in V_{\chi}$ are $A^G$-linearly independent, 
but $A$-linearly dependent for a minimal $m\ge 2$. 
Then $w_1=\sum_{j=2}^m\lambda_jw_j$ for some $\lambda_j\in A^{\times}$. 

Applying $\sigma-\chi(\sigma)$ for each $\sigma\in G$ to both sides of the latter 
equality, we get $\sum_{j=2}^m(\lambda_j^{\sigma}-\lambda_j)\chi(\sigma)w_j=0$, 
and therefore, $\sum_{j=2}^m(\lambda_j^{\sigma}-\lambda_j)w_j=0$. 
By the minimality of $m$, one has $\lambda_j^{\sigma}-\lambda_j=0$ for 
each $\sigma\in G$, so $\lambda_j\in A^G$ for any $j$, contradicting to 
the $A^G$-linear independence of $w_1,\dots,w_m$. \end{proof} 

A $K$-semilinear representation $V$ of $G$ is called {\sl trivial}, if 
the natural map $V^G\otimes_kK\to V$ (injective by Lemma \ref{inject} 
with $\chi=id_V$) is bijective, i.e., if $V$ is isomorphic to 
a direct sum of several copies of $K$ with $G$-action via $\rho$. 

\vspace{4mm}

Now let the semigroup $G$ be totally disconnected and the homomorphism 
$\rho$ be continuous. (We endow any set $H\subseteq\Maps(\Psi_1,\Psi_2)$ 
of mappings between 
sets $\Psi_1$ and $\Psi_2$ with topology, where a base 
$\{U_{\alpha}\}_{\alpha}$ of open subsets of $H$ is indexed by the 
mappings $\alpha$ of finite subsets $S_{\alpha}\subset\Psi_1$ to 
$\Psi_2$ and $U_{\alpha}$ is the set of all elements of $H$ with 
restriction $\alpha$ to $S_{\alpha}$.) 

\begin{lemma} \label{semilin-gener} Let $K$ be a field, 
$G\subseteq\Aut_{\mathrm{field}}(K)$ be a group 
of automorphisms of the field $K$. Let $B$ be such a system of 
open subgroups of $G$ that any open subgroup contains a subgroup 
conjugated, for some $H\in B$, to an open subgroup of finite index 
in $H$. Then the objects $K[G/H]$ for all $H\in B$ form a system 
of generators of the category of smooth $K$-semilinear 
representations of $G$. \end{lemma} 
\begin{proof} Let $V$ be a smooth semilinear representation of $G$. 
Then the stabilizer of any vector $v$ is open, i.e., the stabilizer 
of some vector $v'$ in the $G$-orbit of $v$ admits a subgroup 
commensurable with some $H\in B$. The $K$-linear envelope of the (finite) 
$H$-orbit of $v'$ is a smooth $K$-semilinear representation of $H$, 
so it is trivial, i.e., $v'$ belongs to the $K$-linear envelope of 
the $K^H$-vector subspace fixed by $H$. As a consequence, there is a 
morphism from a finite cartesian power of $K[G/H]$ to $V$, containing 
$v'$ (and therefore, containing $v$ as well) in the image. \end{proof} 

{\sc Example.} For an integer $N\ge 0$ denote by $\binom{\Psi}{N}$ the 
set of all subsets in $\Psi$ of order $N$. Let $S\subseteq\mathbb N$ be 
an infinite set. Let $G=\Sy_{\Psi}$. 
Suppose that the $G$-action on $K$ is faithful. Then the objects 
$K[\binom{\Psi}{N}]\cong\bigwedge_K^NK[\Psi]\cong\Omega^N_{K|k}$, 
$[\{s_1,\dots,s_N\}]\leftrightarrow\prod_{1\le i<j\le N}(s_i-s_j)
[s_1]\wedge\dots\wedge[s_N]\leftrightarrow\prod_{1\le i<j\le N}(s_i-s_j)
ds_1\wedge\dots\wedge ds_N$, for $N\in S$ form a system of generators 
of the category of smooth $K$-semilinear representations of $G$. 

The representation $K[\binom{\Psi}{N}]$ is highly reducible: it will follow from 
Lemma \ref{morphismes-entre-generat} that any finite-dimensional $k$-vector 
space $\Xi$ of symmetric rational functions over $k$ in $N$ variables determines 
a surjective morphism $K[\binom{\Psi}{N}]\to\Hom_k(\Xi,K)$, $[I]\mapsto[Q\mapsto Q(I)]$. 

For any set $\Psi$ and a subset $T\subset\Psi$ denote by (i) $\Sy_{\Psi|T}$ 
the pointwise stabilizer of $T$ in the group $\Sy_{\Psi}$; (ii) 
$\Sy_{\Psi,T}$ the setwise stabilizer of $T$ in the group $\Sy_{\Psi}$, 
i.e., the normalizer of $\Sy_{\Psi|T}$ in $\Sy_{\Psi}$. Then (i) the 
assumptions of Lemma \ref{semilin-gener} hold if $B$ is the set of 
subgroups $\Sy_{\Psi,T}$ for a collection of subsets $T\subset\Psi$ 
with cardinality in $S$, (ii) $K[\binom{\Psi}{N}]$ 
is isomorphic to $K[\Sy_{\Psi}/\Sy_{\Psi,T}]$ for any $T$ of order $N$. 
\qed 

\section{Finiteness conditions on permutation groups} 
\subsection{Roelcke precompact groups} 
\begin{defin} A permutation group $G$ is called {\bf Roelcke precompact} 
if the set $U\backslash G/V$ is finite for any pair 
of open subgroups $U,V\subseteq G$. \end{defin} 

{\sc Example.} A locally compact group is Roelcke precompact if and only 
if it is profinite, since it is a finite union of compact double cosets. 

\begin{lemma} \label{equiv-pseudoprofinite} The following conditions 
on a permutation group $G$ are equivalent: 
\begin{enumerate} \item \label{v} tensor product of any pair of smooth 
finitely generated representations of $G$ over a fixed field is again 
finitely generated, 
\item \label{vii} for any commutative ring $C$ endowed with a smooth 
$G$-action, any associative $C$-algebra $A$ endowed with a smooth 
$G$-action, any smooth finitely generated $A\langle G\rangle$-module 
$V$ and any smooth finitely generated $C\langle G\rangle$-module $W$ 
the $A\langle G\rangle$-module $V\otimes_CW$ is finitely generated, 
\item \label{vi} restriction of any smooth finitely generated 
representation of $G$ to any open subgroup is again finitely generated, 
\item \label{ii} $G$ is Roelcke precompact, 
\item \label{iii} any open subgroup of $G$ is Roelcke precompact. 
\end{enumerate} \end{lemma} 
\begin{proof} The implications (\ref{iii})$\Rightarrow$(\ref{ii}) 
and (\ref{vii})$\Rightarrow$(\ref{v}) are trivial. 

(\ref{ii})$\Rightarrow$(\ref{iii}). For any open subgroup $H$ of $G$ 
and any pair of open subgroups $U,V$ of $H$ the natural map 
$U\backslash H/V\to U\backslash G/V$ is injective. As $U\backslash G/V$ 
is finite, so is $U\backslash H/V$, i.e., $H$ is Roelcke precompact. 

(\ref{v})$\Rightarrow$(\ref{ii}). For any pair of subgroups $U,V$ of 
$G$ one has the following decomposition of the representations of $G$: 
$k[G/U]\otimes_kk[G/V]=\bigoplus_{O\in G\backslash(G/U\times G/V)}k[O]$, 
so $k[G/U]\otimes_kk[G/V]$ is finitely generated if and only if the set 
of orbits $G\backslash(G/U\times G/V)\cong U\backslash G/V$ is finite. 

(\ref{ii})$\Rightarrow$(\ref{vii}). Any smooth finitely generated 
$A\langle G\rangle$-module is a quotient of a finite sum of $A[G/U_i]$ 
for some open subgroups $U_i$ of $G$, so tensor product of a pair of 
finitely generated $A\langle G\rangle$- and $C\langle G\rangle$-modules 
is a quotient of a finite sum of $A\langle G\rangle$-modules 
$A[G/U_i]\otimes_CC[G/V_j]=A[(G/U_i)\times(G/V_j)]$ 
for some open subgroups $U_i,V_j$ of $G$. 

(\ref{ii})$\Leftrightarrow$(\ref{vi}). For any pair of subgroups $U,V$ of 
$G$ one has the following decomposition of the representations of $U$: 
$k[G/V]=\bigoplus_{O\in U\backslash G/V}k[O]$, so $k[G/V]$ is a finitely 
generated representation of $U$ if and only if the set of orbits 
$U\backslash G/V$ is finite. In particular, this shows implication 
(\ref{vi})$\Rightarrow$(\ref{ii}). 

Any smooth finitely generated representations of $G$ is a quotient of 
a finite sum of $k[G/V_i]$ for some open subgroups $V_i$ of $G$, so 
restriction to an open subgroup $U$ of $G$ of a finitely generated 
representation of $G$ is a quotient of a finite sum of representations 
$k[G/V_j]$ for some open subgroups $V_j$ of $G$. This proves (\ref{vi}) 
if we assume (\ref{ii}). \end{proof} 

\begin{lemma} \label{ascending-chains-opens} Let $G$ be a 
Roelcke precompact group and $U\subseteq G$ be an open subgroup. 
Then the set of subgroups of $G$ containing $U$ is finite. 
In particular, {\rm (i)} any ascending chain of open 
subgroups eventually stabilizes, {\rm (ii)} any open subgroup is 
contained in a maximal proper subgroup of $G$, {\rm (iii)} any 
open subgroup is of finite index in a subgroup with no finite 
extensions. \end{lemma} 
\begin{proof} There is a bijection between subgroups of $G$ 
containing $U$ and certain subsets of $U\backslash G/U$. However, 
the set of subsets of the finite set $U\backslash G/U$ is finite. 
\end{proof} 

\begin{lemma} \label{lin-fin} Let $G$ be a Roelcke precompact group. 
Then for any open subgroup $U\subseteq G$ the finite group $N_G(U)/U$ 
acts transitively and freely on the finite set $(G/U)^U$. \end{lemma} 
\begin{proof} For any open subgroup $V\subseteq G$ the set 
$(G/U)^V=\{g\in G~|~Vg\subseteq gU\}/U$ is identified with the set 
$\Maps_G(G/V,G/U)$ (which is the semigroup $\End_G(G/V)$ if 
$V=U$) by $[g]\mapsto([h]\mapsto[hg])$, $\varphi\mapsto\varphi([1])$. 
This set is finite, since $(G/U)^V$ embeds into the finite set 
$V\backslash G/U$. 

If $U$ is a proper subgroup of $gUg^{-1}$ for some 
$[g]\in(G/U)^U$ then we get a strictly increasing sequence of subgroups 
in $G$: $U\subsetneqq gUg^{-1}\subsetneqq g^2Ug^{-2}\subsetneqq g^3Ug^{-3}
\subsetneqq\dots$, contradicting Lemma \ref{ascending-chains-opens} (i). 

This means that the natural inclusion $N_G(U)/U\hookrightarrow(G/U)^U$ 
is bijective. \end{proof} 

\section{Injectivity of trivial representations and admissibility} 
For an abelian group $P$ and a set $S$ we denote by $P[S]^{\circ}$ 
the subgroup of the abelian group $P[S]$ consisting of the finite 
formal sums $\sum_ia_i[s_i]$ for all $a_i\in P$, $s_i\in S$ with 
$\sum_ia_i=0$. If $P$ is a module over a ring $A$ then $P[S]^{\circ}$ 
is naturally an $A$-submodule of the $A$-module $P[S]$. 

Recall, that an injection $M\hookrightarrow N$ in an abelian category is 
called an {\sl essential} extension if any subobject of $N$ has a non-zero 
intersection with the image of $M$, cf. \cite[Ch. 6, \S2]{BucurDeleanu}. 

In Definition~\ref{def-S-type} below, a class of groups $G$ (called the groups of 
$\Sy$-type) will be introduced satisfying the assumptions of the following lemma. 
\begin{lemma} \label{trivial-is-injective} Let $G$ be a 
permutation group admitting a base of open subgroups $B$ such 
that for any $V\subseteq U$ in $B$ the group $N_G(V)/V$ is finite 
and the $(N_G(V)/V)$-action on the set $(G/U)^V\supseteq
(N_G(V)U)/U=N_G(V)/(U\cap N_G(V))$ is transitive, i.e., 
$(G/U)^V=(N_G(V)U)/U$. Let $R$ be a $\mathbb Q$-algebra and $M$ be an $R$-module 
considered as a trivial $G$-module. 

Then any essential extension $E$ of the $R[G]$-module $M$ is a 
trivial $G$-module. In particular, if $M$ is an injective $R$-module 
then (when endowed with the trivial $G$-action) it is an injective 
object of the category of smooth left $R[G]$-modules. \end{lemma} 
\begin{proof} Let $E$ be an essential extension of $M$ in the category 
of smooth left $R[G]$-modules. Any element of $E$ spans a smooth 
cyclic $R[G]$-module. Any smooth cyclic $R[G]$-module is isomorphic 
to a quotient of a permutation module $R[G/U]$ for an open subgroup 
$U$ in a base of open subgroups. I claim that the image of 
$R[G/U]^{\circ}$ in $E$ has zero intersection with $E^G$, 
and in particular, with $M$. 

Indeed, suppose that the image $\beta\in E$ of an element 
$\alpha\in R[G/U]^{\circ}$ is fixed by $G$. 

The support of the element $\alpha$, i.e., a finite subset in $G/U$,  
is pointwise fixed by an open subgroup $V\in B$; in particular, $\alpha$ 
is fixed by $V$. Then $g\alpha$ is well-defined for any $g\in N_G(V)/V$. 
As the image $\beta$ of $\alpha$ is fixed by $G$ (and in particular, by $N_G(V)$), 
the image of $\alpha':=\sum_{g\in N_G(V)/V}g\alpha$ in $E$ is $\#(N_G(V)/V)\beta$. 
By one of the assumptions of the Lemma, the support of $\alpha$ is contained in 
a (finite) $(N_G(V)/V)$-orbit, so the same holds for $\alpha'$, unless $\alpha'=0$. 
On the other hand, $\alpha'$ is fixed by $N_G(V)$, and thus, it is a multiple of 
the sum of the elements of a $N_G(V)$-orbit. But $\alpha'\in R[G/U]^{\circ}$, 
so this multiple is zero, and therefore, $\beta=0$. 

This means that the image of $R[G/U]^{\circ}$ in $E$ has no non-zero vectors 
fixed by $G$. Therefore, $E$ is a quotient of a sum of trivial $G$-modules 
$R[G/U]/R[G/U]^{\circ}=R$, i.e., $E$ is a trivial $G$-module. \end{proof} 

\begin{lemma} \label{admissible-under-inj} Let $G$ be a 
Roelcke precompact group, $B$ be a base of open subgroups of $G$ 
and $k$ be a field. Suppose that the trivial representations 
$k$ of any $V\in B$ are injective as smooth $k[V]$-modules. 
Then any smooth finitely generated $k[G]$-module $W$ is admissible, i.e., 
$\dim_kW^V<\infty$ for any $V\in B$. \end{lemma} 
\begin{proof} As $H^0(V,W)$ is a direct summand of the $k[V]$-module $W$ 
for any $V\in B$, the natural map $H^0(V,W)\to H_0(V,W)$ is injective, so 
$\dim_kH^0(V,W)\le\dim_kH_0(V,W)$. The $k[G]$-module $W$ is a quotient 
of $\bigoplus_{i=1}^Nk[G/U_i]$ for some open subgroups $U_i\subseteq G$, 
in particular, $H_0(V,W)$ is a quotient of 
$H_0(V,\bigoplus_{i=1}^Nk[G/U_i])$, and thus, $\dim_kH_0(V,W)\le
\sum_{i=1}^N\dim_kH_0(V,k[G/U_i])=\sum_{i=1}^N\#[V\backslash G/U_i]<\infty$. 
Combining all these inequalities, we get $\dim_kH^0(V,W)<\infty$. 
\end{proof} 

\begin{theorem} \label{admissibility} Let $G$ be a Roelcke precompact 
group such that the group $N_G(V)$ acts transitively on $(G/U)^V$ 
for all $V\subseteq U$ from a base $B$ of open subgroups of $G$ 
and $k$ be a field of characteristic 0. Then any smooth finitely generated 
$k[G]$-module $W$ is admissible (but there exist 
infinite direct sums among admissible $k[G]$-modules). \end{theorem} 
\begin{proof} For any triple $U_2\subseteq U_1\subseteq V$ in $B$ 
the projection $(G/U_1)^{U_1}\to(G/U_2)^{U_1}$ is surjective, so its 
restriction $(V/U_1)^{U_1}\to(V/U_2)^{U_1}$ is surjective as well. 
This means that $N_V(U_1)$ acts transitively on $(V/U_2)^{U_1}$, 
and thus, any open subgroup $V$ of $G$ satisfies assumptions of 
Lemma~\ref{trivial-is-injective}. Then Theorem follows from 
Lemma~\ref{admissible-under-inj}, since any open subgroup $V$ 
of $G$ satisfies assumptions of Lemma \ref{trivial-is-injective}. 
\end{proof} 

\section{Filtered representations and local length-finiteness} 
\begin{lemma} \label{length-upper-bound} 
Let $G$ be a group, $W$ be a $k[G]$-module for a ring $k$, $B$ 
be a partially ordered set. Let $\{U_{\alpha}\}_{\alpha\in B}$ 
be a partially ordered exhausting collection of $k$-submodules 
in $W$: $W=\bigcup_{\alpha}U_{\alpha}$. Let $G_{\alpha}$ be a 
subgroup of the stabilizer in $G$ of the $k$-submodule $U_{\alpha}$. 

Then length of the $k[G]$-module $W$ does not exceed 
$\inf_{\beta}\sup_{\alpha\ge\beta}\mathrm{length}_{k[G_{\alpha}]}
U_{\alpha}$. \end{lemma} 
\begin{proof} Suppose that $n:=\inf_{\beta}\sup_{\alpha\ge\beta}
\mathrm{length}_{k[G_{\alpha}]}U_{\alpha}$ is finite and 
$0=W_{-1}\subsetneqq W_0\subsetneqq W_1\subsetneqq W_2\subsetneqq\dots
\subsetneqq W_n\subseteq W$ is a chain of $k[G]$-submodules in $W$. 
Choose some $e_i\in W_i\smallsetminus W_{i-1}$ and $\beta\in B$ 
such that $e_0,\dots,e_n\in U_{\beta}$. Then $0\subsetneqq W_0\cap 
U_{\beta}\subsetneqq W_1\cap U_{\beta}\subsetneqq W_2\cap U_{\beta}
\subsetneqq\dots\subsetneqq W_n\cap U_{\beta}\subseteq U_{\beta}$ 
is a chain of $k[G_{\beta}]$-submodules in $U_{\beta}$ 
of length $n+1$, contradicting our assumptions. \end{proof} 

\begin{corollary} Let $G$ be a permutation group, $B$ be a base 
of open subgroups of $G$. Let $W$ be a smooth representation of 
$G$ over a field $k$. Then $\mathrm{length}_{k[G]}W\le
\inf\limits_{V\in B}\sup\limits_{H\in B,~H\subseteq V}
\mathrm{length}_{k[N_G(H)]}W^H$. \end{corollary} 
\begin{proof} In Lemma \ref{length-upper-bound} we take 
$G_H=N_G(H)$ and $U_H=W^H$. \end{proof} 

The following result is standard. 
\begin{lemma} \label{simple-prod} Let $C$ be a finite-dimensional $k$-algebra 
for a field $k$ and $A,B$ be associative unital $k$-algebras. Let $M$ be a 
simple $A$-module and $N$ be a simple $B$-module. Suppose that $\End_B(N)=k$. 
Then {\rm (i)} $M\otimes_kN$ is a simple $A\otimes_kB$-module, {\rm (ii)} 
the $A$- (and $A\otimes_kC$-) module $M\otimes_kC$ is of finite 
length. \end{lemma} 
\begin{proof} 
(ii) is a trivial. (i) Fix a non-zero $A\otimes_kB$-submodule in $M\otimes_kN$ 
and a shortest non-zero element $\alpha=\sum_{i=1}^am_i\otimes n_i$ 
in it, i.e., $a\ge 1$ is minimal possible. Clearly, $\alpha$ is 
a generator of $M\otimes_kN$ if $a=1$, so assume that $a>1$. 
If the annihilators of $n_i$ in $B$ are not the same, say 
$\mathrm{Ann}(n_i)$ is not contained in $\mathrm{Ann}(n_j)$, 
then $(1\otimes\xi)\alpha$ is a shorter non-zero element 
for any $\xi\in\mathrm{Ann}(n_i)\smallsetminus\mathrm{Ann}(n_j)$, 
contradicting to the minimality of $a$. There remains the case 
of coincident annihilators of $n_1,\dots,n_a$. In that case 
$B/\mathrm{Ann}(n_1)\stackrel{\cdot n_i}{\longrightarrow}N$ are 
isomorphisms. As $\End_B(N)=k$, all these isomorphisms differ 
by a non-zero multiple, and therefore, the images of the element 
$1\in B/\mathrm{Ann}(n_1)$ under $\cdot n_i$ differ by a 
non-zero multiple as well, i.e., the elements $n_1,\dots,n_a$ 
are proportional, so finally, $a=1$. \end{proof} 

\subsection{$G$-closed subsets} \label{substructures} 
Let $G\subseteq\Sy_{\Psi}$ be a permutation group. For a subset 
$S\subset\Psi$ we call the set $\Psi^{G_S}$ the $G$-{\sl closure} 
of $S$. We say that a subset $S\subset\Psi$ is $G$-{\sl closed} 
if $S=\Psi^{G_S}$. Any intersection $\bigcap_iS_i$ of $G$-closed 
sets $S_i$ is $G$-closed: as $G_{S_i}\subseteq G_{\bigcap_jS_j}$, 
one has $G_{S_i}s=s$ for any $s\in\Psi^{G_{\bigcap_jS_j}}$, so 
$s\in\Psi^{G_{S_i}}=S_i$ for any $i$, and thus, $s\in\bigcap_iS_i$. 
This implies that the subgroup generated by $G_{S_i}$'s is dense 
in $G_{\bigcap_iS_i}$ (and coincides with $G_{\bigcap_iS_i}$ if 
at least one of $G_{S_i}$'s is open). 

The $G$-closed subsets of $\Psi$ form a small concrete category with the 
morphisms being all those embeddings that are induced by elements of $G$. 

For a finite $G$-closed subset $T\subset\Psi$, 
(hiding $G$ and $\Psi$ from notation) set $\Aut(T):=N_G(G_T)/G_T$. 

\subsection{Groups of $\Sy$-type} 
\begin{defin} \label{def-S-type} A Roelcke precompact group $G$ is of 
{\bf $\Sy$-type} if {\rm (i)} the maps $(G/V)^V\to(G/U)^V$ are surjective for 
all $V\subseteq U$ from a base $B$ of open subgroups of $G$, {\rm (ii)} for 
each $U\in B$ the natural projection $N_U(V)\backslash(G/U)^V\to U\backslash 
G/U$ is injective for sufficiently small $V\in B$. \end{defin}

{\sc Remarks.} 1. The condition (i) is the transitivity condition from 
\cite[p.5, \S3.1]{GanLi}; (ii) is the bijectivity condition of \cite[p.5, \S3.2]{GanLi}. 

2. Clearly, (i) any product of groups of $\Sy$-type is again of $\Sy$-type, 
(ii) a locally compact group is of $\Sy$-type if and only if it is profinite 
(as $B$ we take the set of all normal open subgroups of $G$). 

\vspace{4mm}

{\sc Examples.} \label{examples-Sy} The following examples of groups $G$ of 
$\Sy$-type are constructed as the groups of all permutations of an infinite 
set $\Psi$ respecting an extra structure on $\Psi$. Thus, $G$ is a closed 
subgroup of the group $\Sy_{\Psi}$ of all permutations of the plain set 
$\Psi$. As a base $B$ of open subgroups we take the subgroups $G_T$ for 
some exhausting collection of finite $G$-closed subsets $T\subset\Psi$. 
\begin{enumerate}\item $\Psi$ is a plain set, i.e., $G=\Sy_{\Psi}$. 
Then (i) $G/G_T$ is the set of all embeddings $T\hookrightarrow\Psi$, 
(ii) $(G/G_T)^{G_{T'}}$ consists of the embeddings $T\hookrightarrow T'$ (it is 
clear that $\Aut(T'):=N_G(G_{T'})/G_{T'}=\Sy_{T'}$ acts transitively on $(G/G_T)^{G_{T'}}$) 
and (iii) the rule $[\sigma]\mapsto(\sigma(T)\cap T,\sigma|_{\sigma(T)\cap T})$ 
identifies the sets $G_T\backslash G/G_T$ and 
$N_{G_T}(G_{T'})\backslash(G/G_T)^{G_{T'}}$ (if $\#T'\ge\#T$) with the set of 
pairs $(T_0,\iota)$ consisting of a subset $T_0\subseteq T$ and an 
embedding $\iota:T_0\hookrightarrow T$. 
\item \label{GL-fixed} $G$ is the group of automorphisms of an infinite-dimensional 
vector space $\Psi$ over a finite field identical on a marked finite-dimensional 
subspace $V$, so all $G$-closed $T\subset\Psi$ contain $V$. Then (i) $G/G_T$ is 
the set of all linear embeddings $T\hookrightarrow\Psi$ identical on $V$, and 
(ii) $(G/G_T)^{G_{T'}}$ consists of the embeddings $T\hookrightarrow T'$ 
identical on $V$. 
Clearly, (i) $\Aut(T')$ is transitive on $(G/G_T)^{G_{T'}}$, (ii) the natural 
projection $N_{G_T}(G_{T'})\backslash(G/G_T)^{G_{T'}}\to G_T\backslash G/G_T$ 
is injective, at least if $\dim T'\ge 2\dim T$. \item 
The automorphism group of an infinite-dimensional projective space 
$\Psi=\mathbb{P}(\mathbb{F}_q^S)$ over a finite field ${\mathbb F}_q$. 
Then $G/G_T$ consists of the projective embeddings $T\hookrightarrow\Psi$ 
and $(G/G_T)^{G_{T'}}$ consists of the projective embeddings $T\hookrightarrow T'$. 
The conditions of Definition~\ref{def-S-type} are verified in the same way as 
in Example (\ref{GL-fixed}). \end{enumerate} 

\begin{lemma} \label{lin-fin-length} Let $G$ 
be a group of $\Sy$-type. Then, for any left artinian $\mathbb Q$-algebra 
$R$, a smooth $R[G]$-module is finitely generated if and only if it is of 
finite length. In particular, the category of smooth left $R[G]$-modules 
is a locally artinian and locally noetherian Grothendieck category. \end{lemma} 
\begin{proof} Any smooth finitely generated $R[G]$-module is a quotient of a finite 
sum of $R[G]$-modules of type $R[G/U]$ for some subgroups $U$ from 
a fixed base $B$ of open subgroups of $G$. Therefore, it suffices 
to check that the $R[G]$-modules $R[G/U]$ are of finite length. 

Fix a base $B$ of open subgroups of $G$ as in definition of group of 
$\Sy$-type, some $U\in B$ and set $\Psi:=G/U$. As $G$ is a group of 
$\Sy$-type, $\Psi^V\cong H_V/(H_V\cap(U/V))$ is an $H_V$-orbit for 
all $V\in B$, $V\subseteq U$, where $H_V:=N_G(V)/V$, so \begin{equation} 
\label{end-bound}\End_{F[H_V]}(F[\Psi^V])\cong F[\Psi^V]^{H_V\cap(U/V)}
\cong F[(H_V\cap(U/V))\backslash\Psi^V]\end{equation} 
for any finite field extension $F|\mathbb Q$. The natural projection 
$(H_V\cap(U/V))\backslash\Psi^V=(N_G(V)\cap U)\backslash\Psi^V\to 
U\backslash\Psi$ is injective for sufficiently small $V$, and 
therefore, length of the $F[H_V]$-module $F[\Psi^V]$ does not exceed 
\[\dim_F\End_{F[H_V]}(F[\Psi^V])=\dim_FF[(H_V\cap(U/V))
\backslash\Psi^V]\le\dim_FF[U\backslash\Psi]=\#[U\backslash\Psi].\] 
For a fixed $V$, we choose $F$ so that the $F[H_V]$-module $F[\Psi^V]$ 
is a sum of absolutely simple modules. Clearly, length of the 
$R[G]$-module $R[\Psi]$ does not exceed length of the 
$(R\otimes F)[G]$-module $(R\otimes F)[\Psi]$. By Lemma 
\ref{length-upper-bound}, length of the $(R\otimes F)[G]$-module 
$(R\otimes F)[\Psi]$ does not exceed \[\sup_{V\in B}
\mathrm{length}_{(R\otimes F)[H_V]}((R\otimes F)[\Psi^V]).\] 
On the other hand, by Lemma \ref{simple-prod} (i), 
$\mathrm{length}_{(R\otimes F)[H_V]}((R\otimes F)[\Psi^V])=
\mathrm{length}_{(R\otimes F)}(R\otimes F)\cdot
\mathrm{length}_{F[H_V]}(F[\Psi^V])$, so finally, 
$\mathrm{length}_{(R\otimes F)[G]}((R\otimes F)[\Psi])\le
\mathrm{length}_{R\otimes F}(R\otimes F)\cdot\#[U\backslash\Psi]$, 
which is finite by Lemma \ref{simple-prod} (ii). \end{proof} 

\section{Smooth representations: coinduction, simple objects 
and injectives} 
\subsection{Induction and coinduction} Let $A\subset R$ be a pair of associative 
unital topological rings with a base of neighbourhoods of 0 given by a collection 
of left ideals. The restriction functor $R\mbox{-mod}\to A\mbox{-mod}$ admits a 
left adjoint $R\otimes_A(-)$ (induction): $\Hom_A(E,W)=\Hom_R(R\otimes_AE,W)$. 
Let $R\mbox{-mod}^{\mathrm{sm}}$ be the category of left $R$-modules 
such that any element is annihilated by an open left ideal. 
Then the restriction functor induces a functor 
$R\mbox{-mod}^{\mathrm{sm}}\to A\mbox{-mod}^{\mathrm{sm}}$ 
admitting a right adjoint (coinduction), sending a smooth 
$A$-module $E$ to the smooth part of the $R$-module $\Hom_A(R,E)$: 
$\Hom_A(W,E)=\Hom_R(W,\Hom_A(R,E)^{\mathrm{sm}})$. Clearly, if a 
smooth $A$-module $E$ is injective then $\Hom_A(R,E)^{\mathrm{sm}}$ 
is an injective smooth $R$-module. In particular, if the restriction 
to $A$ of a smooth $R$-module $W$ is injective the adjunction 
morphism $W\to\Hom_A(R,W)^{\mathrm{sm}}$ gives an embedding into 
an injective smooth $R$-module. 

Our only examples of such topological rings will be skew group 
rings $A\langle G\rangle$ with the base of 
open left ideals $I_U$ indexed by a base $B$ of open subgroups $U$ 
of $G$, where $I_U$ is generated by elements $u-1$ for all $u\in U$. 

\begin{lemma} \label{finite-coset} For any group $G$, any finite collection 
$H_1,\dots,H_N$ of subgroups of $G$ of infinite index and any finite collection of 
$\xi_{ij}\in G$ one has $G\neq\bigcup_{i=1}^N(\coprod_{j=1}^M\xi_{ij}H_i)$. \end{lemma} 
\begin{proof} We proceed by induction on $N\ge 1$, the case $N=1$ being trivial. 
Suppose that $G=\bigcup_{i=1}^N(\coprod_{j=1}^M\xi_{ij}H_i)$ for some 
$\xi_{ij}\in G$. If $H_s$ and $H_t$ are commensurable for some $s\neq t$ 
then $G=\bigcup_{i\neq t}(\coprod_{j=1}^{M[H_t:H_s\cap H_t]}\xi_{ij}'H_i)$ 
for some $\xi_{ij}'\in G$, contradicting the induction assumption, so 
we further assume that $H_i$ are pairwise non-commensurable. 

On a set of pairwise non-conmensurable subgroups of $G$ define the following relation 
$H_1\preceq H_2$ if $[H_1:H_1\cap H_2]<\infty$. This is a partial order: if $H_1\preceq H_2$ 
and $H_2\preceq H_3$ then $[H_1\cap H_2:H_1\cap H_2\cap H_3]\le[H_2:H_2\cap H_3]<\infty$ 
(since the natural map $H_1\cap H_2/H_1\cap H_2\cap H_3\to H_2/H_2\cap H_3$ is injective). 

Let $H_s$ be maximal among $H_i$ with respect to $\preceq$. Fix $\xi\in G$ 
outside the set $\xi_{s1}H_s\coprod\cdots\coprod\xi_{sM}H_s$. Then 
$\xi H_s\subseteq\bigcup_{i\neq s}(\coprod_{j=1}^M\xi_{ij}H_i)$. Omitting those $\xi_{ij}$ 
for which $\xi H_s\cap\xi_{ij}H_i=\varnothing$ and replacing appropriately 
the remaining $\xi_{ij}$'s, we may assume that $\xi^{-1}\xi_{ij}\in H_s$. 

Finally, $H_s=\bigcup_{i\neq s}(\coprod_{j=1}^M\xi^{-1}\xi_{ij}H_i\cap H_s)$, 
contradicting induction assumption. \end{proof}

\begin{lemma} \label{con} Let $G$ be a group and $E$ be an abelian group. For any subgroup 
$H\supseteq G$ define the additive map \[\kappa_H:E[G/H]\to\Maps(G,E)\quad\text{{\rm by}}
\quad e[g]\mapsto\left(\xi\mapsto\left\{\begin{array}{ll}e&\text{{\rm if} $\xi^{-1}\in[g]$}\\ 
0&\text{{\rm otherwise}}\end{array}\right.\right)\quad\text{{\rm for any $e\in E$ 
and $g\in G$}}.\] Let $\{H_i\}$ be a collection of pairwise non-commensurable subgroups 
of $G$. 

Then the map $\sum_i\kappa_{H_i}:\bigoplus_iE[G/H_i]\to\Maps(G,E)$ is $G$-equivariant (if 
$\Maps(G,E)$ is endowed with the standard $G$-action: $\varphi^g(\xi):=\varphi(\xi g)$), 
injective and factors through $\Maps(H\backslash G,E)$. \end{lemma} 
\begin{proof} As $\xi^{-1}\in[g]$ if and only if $(h\xi)^{-1}\in[g]$ for any $h\in H$, 
$\kappa_H([g])(h\xi)=\kappa_H([g])(\xi)$ for any $h\in H$, i.e., $\kappa_H$ factors 
through $\Maps(H\backslash G,E)$. Namely, $\kappa_H$ transforms $e[g]$ to the 
delta-function with non-zero value $e$ supported on $[g^{-1}]\in H\backslash G$. 

Suppose that $\sum_i\kappa_{H_i}$ is not injective, so its kernel contains an element 
$\sum_{t=1}^N\alpha_t$ for some non-zero $\alpha_t\in E[G/H_{i_t}]$. Let $H_{i_s}$ be 
a maximal subgroup with respect to the partial order $\preceq$ defined 
in the proof of Lemma \ref{finite-coset}. Then looking at the support of both sides 
of the equality $-\kappa_{H_{i_s}}(\alpha_s)=\sum_{t\neq s}\kappa_{H_{i_t}}(\alpha_t)$, 
we get an inclusion $H_{i_s}\subseteq\bigcup_{t\neq s}(\coprod_jH_{i_t}\xi_{tj})$ into 
a finite union of cosets. But this contradicts to Lemma \ref{finite-coset}, 
so $\sum_i\kappa_{H_i}$ is injective. 

To check the $G$-equivariantness ($\kappa_H(e[g'g])(\xi)=\kappa_H(e[g])(\xi g')$), it 
suffices to note that $\xi^{-1}\in[g'g]$ if and only if $(\xi g')^{-1}\in[g]$. \end{proof}

\begin{lemma} \label{coinduction-gener} Let $G$ be a Roelcke precompact group and $H\subseteq G$ 
be an open subgroup. Let $R$ be an associative unitary ring endowed with the trivial $G$-action 
and $E$ be a left $R$-module considered as a trivial $H$-module. Suppose that there is a 
collection $\{G_{\Lambda}\}$ of pairwise non-commensurable open subgroups of $G$ such that 
$\sum_{\Lambda}\#\{\text{{\rm finite $U$-orbits in $G/G_{\Lambda}$}}\}=\#(H\backslash G/U)$ 
for any $U$ from a base of open subgroups of $G$. Then the left $R[G]$-module $W$ coinduced 
by $E$ {\rm (}i.e., the smooth part of $\Hom_{R[H]}(R[G],E)=\Maps_H(G,E)${\rm )} is 
isomorphic to $\bigoplus_{\Lambda}E[G/G_{\Lambda}]$. \end{lemma} 
\begin{proof} Lemma \ref{con} provides a natural injective morphism 
$\sum_{\Lambda\subseteq J}\kappa_{\Lambda}:
\bigoplus_{\Lambda\subseteq J}E[G/G_{\Lambda}]\to W$. To check its surjectivity, we verify 
the surjectivity of the induced maps $\bigoplus_{\Lambda\subseteq J}E[G/G_{\Lambda}]^U
\to W^U=\Maps(H\backslash G/U,E)$, where $U$ runs over a base of open subgroups of $G$. 

As $W^U$ is spanned by the delta-functions on $G_J\backslash G/U$, it suffices to check 
the surjectivity in the case $E=\mathbb Z$, which in turn is equivalent to the cases where 
$E$ runs over all prime fields. If $E$ is a field then the surjectivity is equivalent 
to the coincidence of dimensions of the source and the target, i.e., to the equality 
$\sum_{\Lambda}\#\{\text{finite $U$-orbits in $G/G_{\Lambda}$}\}=\#(H\backslash G/U)$. 
\end{proof} 

\begin{proposition} \label{coinduction} Let $G$ be the group of automorphisms 
of a countable $\Psi$, which is either a set or a vector space over a finite field, 
fixing a finite subset of $\Psi$. Let $R$ be an associative unitary ring endowed 
with the trivial $G$-action, $J\subset\Psi$ be a $G$-closed finite subset and $E$ be 
a left $R$-module considered as a trivial $G_J$-module. Then the left $R[G]$-module 
coinduced by $E$ is isomorphic to $\bigoplus_{\Lambda}E[G/G_{\Lambda}]$, where 
$\Lambda$ runs over the $G$-closed subsets of $J$. \end{proposition} 
\begin{proof} As $G_{\Lambda}$ are pairwise non-commensurable, by Lemma \ref{coinduction-gener}, 
we only need to construct a bijection $\mu:H\backslash G/U\stackrel{\sim}{\longrightarrow}
\coprod_{\Lambda}\{\text{{\rm finite $U$-orbits in $G/G_{\Lambda}$}}\}$, where 
$H=G_J$ and $U=G_T$ for $T$ running over the $G$-closed subsets of $\Psi$. 
To each element $[\sigma]\in H\backslash G/U$ we associate the subgroup 
$\langle H,\sigma U\sigma^{-1}\rangle$ (generated by $H$ and $\sigma U\sigma^{-1}$) 
as $G_{\Lambda(\sigma)}$ and the class $[\sigma^{-1}]_{\Lambda}$ of 
$\sigma^{-1}$ in $G/G_{\Lambda(\sigma)}$. Clearly, these $G_{\Lambda(\sigma)}$ 
and $[\sigma^{-1}]_{\Lambda}$ are well-defined and, in particular, 
$[\sigma^{-1}]_{\Lambda(\sigma)}$ is fixed by $U$. This gives rise to 
a map $\eta:H\backslash G/U\to\coprod_{\Lambda}(G/G_{\Lambda})^U$. 
According to \S\ref{substructures}, $G_{\Lambda(\sigma)}=G_{J\cap\sigma(T)}$. 
It is easy to see that (i) $\eta$ is bijective and (ii) any finite $G_T$-orbit 
in $G/G_{\Lambda}$ consists of a single element, so we are done. \end{proof} 

\subsection{Growth estimates} 
Let $G\subseteq\Sy_{\Psi}$ be a permutation group such that for any integer 
$N\ge 0$ the $G$-closed subsets of length $N$ form a non-empty $G$-orbit. 
For each integer $N\ge 0$ fix a $G$-closed subset $\Psi_N\subset\Psi$ 
of length $N$, i.e., $N$ is the minimal cardinality of the subsets 
$S\subset\Phi$ such that $\Psi_N$ is the $G$-closure of $S$. 

For a division ring endowed with a $G$-action and an 
$A\langle G\rangle$-module $M$ define a function 
$d_M:{\mathbb Z}_{\ge 0}\to{\mathbb Z}_{\ge 0}\sqcup\{\infty\}$ by 
$d_M(N):=\dim_{A^{G_{\Psi_N}}}(M^{G_{\Psi_N}})$. 
\begin{lemma} \label{growth} Let $G$ be either $\Sy_{\Psi}$ (and then $q:=1$) or the group 
of automorphisms of an $\mathbb F_q$-vector space $\Psi$ fixing a subspace of finite dimension 
$v\ge 0$. Let $A$ be a division ring endowed with 
a $G$-action. If $0\neq M\subseteq A[G/G_{\Psi_n}]$ for some $n\ge 0$ 
then $d_M$ grows as a $q$-polynomial of degree $n$: 
\[\frac{1}{d_n(n)}([N]_q-[n+m-1]_q)^n\le\frac{d_{m+n}(N)}{d_m(N)d_n(n)}
\le d_M(N)\le q^{vn}d_n(N)\le q^{vn}[N]_q^n\] 
for some $m\ge 0$, where $[s]_q:=\#\Psi_s$ and $d_n(N)$ is the number of embeddings 
$\Psi_n\hookrightarrow\Psi_N$ induced by elements of $G$, 
which is $([N]_q-[0]_q)\cdots([N]_q-[n-1]_q)$. \end{lemma} 
\begin{proof} As $M^{G_{\Psi_N}}\subseteq A[N_G(G_{\Psi_N})/(N_G(G_{\Psi_N})\cap G_{\Psi_n})]$ 
and (by Lemma \ref{inject}) $A\otimes_{A^{G_{\Psi_N}}}M^{G_{\Psi_N}}\to M\subseteq A[G/G_{\Psi_n}]$ 
is injective, there is a natural inclusion \[A\otimes_{A^{G_{\Psi_N}}}M^{G_{\Psi_N}}\hookrightarrow 
A[N_G(G_{\Psi_N})/(N_G(G_{\Psi_N})\cap G_{\Psi_n})]=A[\Aut(\Psi_N)/\Aut(\Psi_N|\Psi_n)],\] 
if $n\le N$. (Here $\Aut(\Psi_N|\Psi_n)$ denotes the automorphisms of $\Psi_N$ identical on 
$\Psi_n$.) Then one has $d_M(N)\le\#(\Aut(\Psi_N)/\Aut(\Psi_N|\Psi_n))=q^{vn}d_n(N)$. 
The lower bound of $d_M(N)$ is given by the number of $G$-closed 
subsets in $\Psi_N$ with length-0 intersection with $\Psi_m$. 
Indeed, for any non-zero element $\alpha\in M\subseteq A[G/G_{\Psi_n}]$ there exist 
an integer $m\ge 0$ and elements $\xi,\eta\in G$ such that $\xi\alpha$ is congruent 
to $\sum_{\sigma\in\Aut(\Psi_n)}b_{\sigma}\eta\sigma$ for some non-zero collection 
$\{b_{\sigma}\in A\}_{\sigma\in\Aut(\Psi_n)}$ modulo monomorphisms whose 
images have intersection of positive length with a fixed finite $\Psi_m$. \end{proof} 

Let $q$ be either 1 or a primary integer. Let $S$ be a plain set if $q=1$ 
and an $\mathbb F_q$-vector space if $q>1$. For each integer $s\ge 0$, we 
denote by $\binom{S}{s}_q$ the set of subobjects of $S$ ($G$-closed subsets of $\Psi$, 
if $S=\Psi$, where $G=\Sy_{\Psi}$ if $q=1$ and $G=\GL_{\mathbb F_q}(\Psi)$ if $q>1$) of length $s$. 
In other words, $\binom{S}{s}_1:=\binom{S}{s}$, while $\binom{S}{s}_q$ is 
the Grassmannian of the $s$-dimensional subspaces in $S$ if $q>1$. 
\begin{corollary} \label{def-binom} Let $G$ be either $\Sy_{\Psi}$ (and then $q:=1$) 
or the group of automorphisms of an $\mathbb F_q$-vector space $\Psi$ fixing a 
finite-dimensional subspace of $\Psi$. Let $A$ be a division 
ring endowed with a $G$-action. Let $\Xi$ be a finite 
subset in $\Hom_{A\langle G\rangle}(A[G/G_T],A[G/G_{T'}])$ for 
some finite $G$-closed $T'\subsetneqq T\subset\Psi$. Then 
\begin{enumerate}\item \label{any-is-essential} any non-zero 
$A\langle G\rangle$-submodule of $A[\binom{\Psi}{m}_q]$ is essential; 
\item\label{level} there are no nonzero isomorphic 
$A\langle G\rangle$-submodules in $A[G/G_T]$ and $A[G/G_{T'}]$; 
\item \label{kernel-of-finite} the common kernel 
$V_{\Xi}$ of all elements of $\Xi$ is an essential 
$A\langle G\rangle$-submodule in $A[G/G_T]$. \end{enumerate} \end{corollary} 
\begin{proof} (\ref{any-is-essential}) follows from 
the lower growth estimate of Lemma \ref{growth}. 

(\ref{level}) follows immediately from Lemma \ref{growth}. 

(\ref{kernel-of-finite}) Suppose that there exists a nonzero submodule 
$M\subseteq A[G/G_T]$ such that $M\cap V_{\Xi}=0$. Then 
restriction of some $\xi\in\Xi$ to $M$ is nonzero. If $\xi|_M$ is not 
injective, replacing $M$ with $\ker\xi\cap M$, we can assume that 
$\xi|_M=0$. In other words, we can assume that restriction to $M$ of 
any $\xi\in\Xi$ is either injective or zero. In particular, restriction 
to $M$ of some $\xi\in\Xi$ is injective, i.e. $\xi$ embeds $M$ into 
$A[G/G_{T'}]$, contradicting to (\ref{level}). \end{proof} 

\subsection{More notations and a description of injective cogenerators} 
Let $G\subseteq\Sy_{\Psi}$ be a permutation group. For a finite $G$-closed subset 
$T\subset\Psi$, denote by $V^{\Psi}_T$ the common 
kernel of all morphisms of $G$-modules $\pi:\mathbb Z[G/G_T]\to\mathbb Z[G/G_{T'}]$ 
for all proper $G$-closed subsets $T'\subsetneqq T$. 

{\sc Remarks.} 1. In this definition it suffices to consider only maximal $T'$ and only 
the morphisms $\pi$ induced by the natural projections $G/G_T\to G/G_{T'}$, 
since the elements of the $\Aut(T')$-orbits of such $\pi$'s generate 
$\Hom_G(\mathbb Z[G/G_T],\mathbb Z[G/G_{T'}])$. 

2. Clearly, any endomorphism of $\mathbb Z[G/G_T]$ preserves $V^{\Psi}_T$, 
so $V^{\Psi}_T$ is an $\End_G(\mathbb Z[G/G_T])$-module. 

3. For any subset $S\subseteq T$ such that $G_T=\bigcap_{t\in S}G_{\{t\}}$ the 
induced map $\mathbb Z[G/G_T]\to\mathbb Z[\prod_{t\in S}G/G_{\{t\}}]$ is injective and 
identifies $V^{\Psi}_T$ with $\mathbb Z[G/G_T]\cap\bigotimes_{t\in S}V^{\Psi}_{\{t\}}
\subseteq\bigotimes_{t\in S}\mathbb Z[G/G_{\{t\}}]$. 

\vspace{4mm} 

Let $k$ be a field of characteristic zero and $\Pi(T)$ be the set of isomorphism 
classes of simple $k[\Aut(T)]$-modules. (E.g., in the case $G=\Sy_{\Psi}$ the 
elements of $\Pi(T)$, called {\sl Specht} $k[\Sy_T]$-modules, correspond to the 
partitions of cardinality of $T$ and they are absolutely irreducible.) 
For each $\alpha\in\Pi(T)$ we choose its representative $\Gamma_{T,\alpha}$ and set 
$D_{\alpha}:=\End_{k[\Aut(T)]}(\Gamma_{T,\alpha})$, so $k[\Aut(T)]\cong\bigoplus
_{\alpha\in\Pi(T)}\End_{D_{\alpha}}(\Gamma_{T,\alpha})$. Then there are canonical 
decompositions of $k[G\times\Aut(T)]$-modules \begin{gather*}k[G/G_T]=\bigoplus
_{\alpha\in\Pi(T)}S^{\Psi}_{T,\alpha}\otimes_{D_{\alpha}}\Gamma_{T,\alpha}
\quad\mbox{and}\quad V^{\Psi}_T\otimes k=\bigoplus_{\alpha\in\Pi(T)}
V^{\Psi}_{T,\alpha}\otimes_{D_{\alpha}}\Gamma_{T,\alpha},\quad\mbox{where}\\ 
S^{\Psi}_{T,\alpha}:=\Hom_{k[\Aut(T)]}(\Gamma_{T,\alpha},k[G/G_T])
\quad\mbox{and}\quad V^{\Psi}_{T,\alpha}:=
\Hom_{k[\Aut(T)]}(\Gamma_{T,\alpha},V^{\Psi}_T\otimes k).\end{gather*} 

The following results may be well-known to the experts. 
\begin{theorem} \label{injective-cogenerators} Let $G$ be the group of automorphisms 
of a countable $\Psi$, which is either a set or a vector space over a finite field, fixing 
a finite subset of $\Psi$. 
Let $k$ be a field of characteristic zero. Then 
the following holds. \begin{enumerate} \item \label{} The $k[G]$-module 
$V^{\Psi}_{T,\alpha}$ is simple for any $T$ and any $\alpha\in\Pi(T)$. In particular, 
the $k[G]$-module $V^{\Psi}_T\otimes k$ is semisimple. Two $k[G]$-modules 
$V^{\Psi}_{T,\alpha}$ and $V^{\Psi}_{T',\alpha'}$ are isomorphic if and only if 
$T'=g(T)$ for some $g\in G$ and $g$ transforms $\alpha$ to $\alpha'$. 
\item \label{simple-in-V} Any simple smooth $k[G]$-module $W$ is isomorphic to 
$V^{\Psi}_{T,\alpha}$ for some finite $G$-closed subset $T\subset\Psi$ and $\alpha\in\Pi(T)$. 
\item \label{socle} The minimal essential submodule of any module $M$ of finite 
length coincides with its maximal semisimple submodule (the socle), which 
is $V^{\Psi}_T\otimes k$ in the case $M=k[G/G_T]$. 
\item \label{inj-hulls-simple} The smooth $k[G]$-module $S^{\Psi}_{T,\alpha}$ is 
an injective hull of $V^{\Psi}_{T,\alpha}$. 
\item \label{all-indec-inj} The $k[G]$-module $S^{\Psi}_{T,\alpha}$ is indecomposable 
for any $T$ and any $\alpha\in\Pi(T)$. Two $k[G]$-modules $S^{\Psi}_{T,\alpha}$ 
and $S^{\Psi}_{T',\alpha'}$ are isomorphic if and only if $T'=g(T)$ for some $g\in G$ 
and $g$ transforms $\alpha$ to $\alpha'$. Any indecomposable injective smooth 
$k[G]$-module is isomorphic to $S^{\Psi}_{T,\alpha}$ for some $T\subset\Psi$ and $\alpha\in\Pi(T)$. 
\item \label{no-morphisms} 
The simple subquotients of the $k[G]$-module $k[G/G_T]$ are isomorphic to $V^{\Psi}_{T',\alpha}$ 
for all finite $G$-closed subsets $T'\subset T$ and all $\alpha\in\Pi(T')$. \end{enumerate}
\end{theorem} 
\begin{proof} Let us show first that (i) any smooth simple $k[G]$-module 
$W$ can be embedded into $k[G/G_T]$ for a finite subset $T\subset\Psi$, 
(ii) the smooth $k[G]$-module $k[G/G_T]$ is injective for any finite $T$. 
Let $J\subset\Psi$ be a finite $G$-closed subset such that $W^{G_J}\neq 0$. 
By Lemma~\ref{trivial-is-injective}, the trivial $k[G_J]$-module $k$ is 
injective, so (i) any surjection of $k$-vector spaces $W^{G_J}\to k$ extends 
to a surjection of $k[G_J]$-modules $\pi:W\to k$, (ii) the smooth part of 
$\Maps_{G_J}(G,k)$ is injective. Then $\pi$ induces a non-trivial, and thus 
injective, morphism $W\to\Maps_{G_J}(G,k)$, $w\mapsto[g\mapsto\pi(gw)]$. By 
Proposition~\ref{coinduction}, this implies (i) that $W$ can be embedded into 
$k[G/G_{\Lambda}]$ for a subset $\Lambda\subseteq J$, (ii) the injectivity of $k[G/G_T]$. 

To see that any smooth simple $k[G]$-module $W$ can be embedded 
into $V^{\Psi}_T\otimes k$ for a finite $G$-closed subset $T\subset\Psi$, 
we embed $W$ into $k[G/G_T]$ for a minimal $T$. Then $W$ is, in fact, 
embedded into $V^{\Psi}_T\otimes k\subseteq k[G/G_T]$: otherwise $W$ 
embeds into $k[M^{\Psi}_{T'}]$ for a proper $G$-closed $T'\subset T$. 

For any pair of finite $G$-closed subsets $T\subseteq T'\subset\Psi$, let us construct natural 
isomorphisms \[k[\Aut(T)]\stackrel{\sim}{\longrightarrow}\End_{k[G]}(k[G/G_T]]),
\quad\Hom_{k[G]}(k[G/G_T],k[G/G_{T'}])\stackrel{\sim}{\longrightarrow}\Hom_{k[G]}
(V^{\Psi}_T\otimes k,V^{\Psi}_{T'}\otimes k),\] where the latter space vanishes if $T\subsetneqq T'$. 
As $k[\Aut(T)]\cong\bigoplus_{\alpha\in\Pi(T)}\End_{D_{\alpha}}(\Gamma_{T,\alpha})$, 
this will imply that any endomorphism of $k[G/G_T]$ or of $V^{\Psi}_T\otimes k$ is 
a composition of a projector and an automorphism. One has $\End_{k[G]}(k[G/G_T])\cong 
k[G/G_T]^{G_T}=k[N_G(G_T)/G_T]=k[\Aut(T)]$. The restriction $\Hom_{k[G]}(k[G/G_T],k[G/G_{T'}])
\to\Hom_{k[G]}(V^{\Psi}_T\otimes k,k[G/G_{T'}])$ is surjective, since $k[G/G_{T'}]$ is injective.
On the other hand, it is injective, since any $\varphi\in\Hom_{k[G]}(k[G/G_T],k[G/G_{T'}])$ 
vanishing on $V^{\Psi}_T\otimes k$ factors through a $k[G]$-submodule of a finite 
cartesian power of $k[G/G_{T''}]$ for a proper $G$-closed $T''\subsetneqq T$, while 
$\Hom_{k[G]}(k[G/G_{T''}],k[G/G_{T'}])=0$. 
It follows from definition of $V^{\Psi}_T$ that $\Hom_{k[G]}(V^{\Psi}_T\otimes k,k[G/G_{T'}])
=\Hom_{k[G]}(V^{\Psi}_T\otimes k,V^{\Psi}_{T'}\otimes k)$. 

Next, let us show that the $k[G]$-module $V^{\Psi}_T\otimes k$ is semisimple. 
Assuming the contrary, $V^{\Psi}_T\otimes k$ contains an indecomposable 
$W'\supset W_1\neq 0$ with simple $W'/W_1$. There is an embedding 
$\iota:W'/W_1\hookrightarrow k[G/G_{T'}]$ for some finite $G$-closed 
$T'\subset\Psi$. As $k[G/G_{T'}]$ is injective, $\iota$ extends to $k[G/G_T]/W_1
\longrightarrow k[G/G_{T'}]$, so gives rise to a non-zero morphism $k[G/G_T]
\stackrel{\varphi}{\longrightarrow}k[G/G_{T'}]$, and thus, $\#T\ge\#T'$ if $T$ is 
$G$-closed. If $\#T'<\#T$ then $\varphi|_{V^{\Psi}_T\otimes k}=0$, 
so $\varphi|_{W'}=0$, which is contradiction. Thus, $\#T=\#T'$, which means that 
$\varphi$ can be considered as an endomorphism in $\End_{k[G]}(k[G/G_T])$. As 
$\varphi$ is a composition of a projector and an automorphism, we may assume that 
$\varphi$ is a projector. Therefore, $W'\cong\varphi(W')\oplus(id-\varphi)(W')$, 
so $\varphi$ embeds $W'/W_1$ into $\varphi(W')$ and $id-\varphi$ embeds 
$W'/W_1$ into $W_1\cong(id-\varphi)(W')$, so $(\varphi,id-\varphi):
(W'/W_1)\oplus W_1\to\varphi(W')\oplus(id-\varphi)(W')$ is injective. 
As the source and the target of $(\varphi,id-\varphi)$ are of the 
same finite length, $(\varphi,id-\varphi)$ is an isomorphism. 
In particular, $(W'/W_1)\oplus W_1\cong W'$ contradicting to the indecomposability of $W'$. 

Our next task is to show that (i) the $k[G]$-module $V^{\Psi}_{T,\alpha}$ is simple 
for any $T$ and any $\alpha\in\Pi(T)$ and (ii) two $k[G]$-modules $V^{\Psi}_{T,\alpha}$ 
and $V^{\Psi}_{T',\alpha'}$ are isomorphic if and only if $T'=g(T)$ for some 
$g\in G$ and $g$ transforms $\alpha$ to $\alpha'$. As we already know, $V^{\Psi}_T$ 
and $V^{\Psi}_{T'}$ are semisimple and $\Hom_{k[G]}(V^{\Psi}_T,V^{\Psi}_{T'})=0$ 
for any finite $G$-closed subsets $T\subsetneqq T'\subset\Psi$. This implies that 
$\Hom_{k[G]}(V^{\Psi}_T,V^{\Psi}_{T'})=0$ if $T$ and $T'$ are not of same length. 
Then it suffices to check that $\Hom_{k[G]}(V^{\Psi}_{T,\alpha},V^{\Psi}_{T,\alpha'})$ 
is a division ring if $\alpha=\alpha'$ or is zero otherwise. One has 
$\End_{k[G]}(V^{\Psi}_T\otimes k)=\bigoplus_{\alpha\in\Pi(T)}\bigoplus_{\beta\in\Pi(T)}
\Hom_{k[G]}(V^{\Psi}_{T,\alpha}\otimes_{D_{\alpha}}\Gamma_{T,\alpha},
V^{\Psi}_{T,\beta}\otimes_{D_{\beta}}\Gamma_{T,\beta})$. The natural map 
$k[\Aut(T)]\cong\bigoplus_{\alpha\in\Pi(T)}\End_{D_{\alpha}}(\Gamma_{T,\alpha})\to
\bigoplus_{\alpha\in\Pi(T)}\End_{k[G]}(V^{\Psi}_{T,\alpha}\otimes_{D_{\alpha}}\Gamma_{T,\alpha})$ 
is bijective, and therefore, since the composition of this map with the inclusion into 
$\End_{k[G]}(V^{\Psi}_T\otimes k)$ is also a bijection, the complementary summand 
$\bigoplus_{\alpha,\beta\in\Pi(T),~\alpha\neq\beta}\Hom_{k[G]}(V^{\Psi}_{T,\alpha}
\otimes_{D_{\alpha}}\Gamma_{T,\alpha},V^{\Psi}_{T,\beta}\otimes_{D_{\beta}}\Gamma_{T,\beta})$ 
of $\End_{k[G]}(V^{\Psi}_T\otimes k)$ vanishes. 

The coincidence of the minimal essential submodule, the maximal 
semisimple submodule and the essential semisimple submodule is a general statement: 
any simple subobject of an object $M$ is contained in any essential subobject 
of $M$, while any non-zero object of finite length contains a simple subobject. 
By Corollary \ref{def-binom} (\ref{kernel-of-finite}), $V^{\Psi}_T\otimes k$ 
is essential in $k[G/G_T]$. As $V^{\Psi}_T\otimes k$ is semisimple, we are done. 

As we already know, $k[G/G_T]$ is injective, so its direct summand $S^{\Psi}_{T,\alpha}$ 
is injective as well. As $V^{\Psi}_T$ is the socle of $k[G/G_T]$, the $k[G]$-module 
$V^{\Psi}_{T,\alpha}$ is the socle of $S^{\Psi}_{T,\alpha}$. As the socle of 
$S^{\Psi}_{T,\alpha}$ is simple, $S^{\Psi}_{T,\alpha}$ is indecomposable. 

By the evident induction on $\# T$, the simple 
subquotients of the $k[G]$-module $k[G/G_T]$ are isomorphic to the 
direct summands of the $k[G]$-modules $V^{\Psi}_{T'}\otimes k$ for all 
$T'\subseteq T$. As any simple $k[G]$-module is a quotient of $k[G/G_T]$ 
for an appropriate $T\subset\Psi$, we get the claim. \end{proof} 

\section{Noetherian properties of smooth semilinear representations 
of $\Sy_{\Psi}$} 
\begin{lemma} \label{indecomp} Let $G$ be a group acting on a field $K$. 
Let $U$ be a subgroup of $G$ such that $(G/U)^U=\{[U]\}$ {\rm (}i.e., 
$\{g\in G~|~gU\subseteq Ug\}=U${\rm )} and $[U:U\cap(gUg^{-1})]=\infty$, 
unless $g\in U$. Then $\End_{K\langle G\rangle}(K[G/U])=K^U$ is a field, 
so $K[G/U]$ is indecomposable. \end{lemma} 
\begin{proof} Indeed, $\End_{K\langle G\rangle}(K[G/U])=(K[G/U])^U=K^U
\oplus(K[(G\smallsetminus U)/U])^U$. As $U(gUg^{-1})$ 
consists of $[U:U\cap(gUg^{-1})]$ classes in $G/(gUg^{-1})$, 
we see that $(K[(G\smallsetminus U)/U])^U=0$. \end{proof} 

{\sc Examples.} 1. Let $\Psi$ be an infinite set, possibly endowed with 
a structure of a projective space. Let $G$ be the group of automorphisms 
of $\Psi$, respecting the structure, if any. Let $J$ be the $G$-closure 
of a finite subset in $\Psi$, i.e., a finite subset or a finite-dimensional 
subspace. Let $U$ be the stabilizer of $J$ in $G$. Then $G/U$ is identified 
with the set of all $G$-closed subsets in $\Psi$ of the same length as $J$. 

2. By Lemma \ref{indecomp}, $K[G/U]$ is indecomposable in the following 
examples: \begin{enumerate} \item $G$ is the group of projective 
automorphisms of an infinite projective space $\Psi$ (i.e., either 
$\Psi$ is infinite-dimensional, or $\Psi$ is defined over an infinite 
field), $U$ is the setwise stabilizer in $G$ of a finite-dimensional 
subspace $J\subseteq\Psi$. Then $G/U$ is identified with the 
Grassmannian of all subspaces in $\Psi$ of the same dimension as $J$.  
\item $G$ is the group of permutations of an infinite set $\Psi$, 
$U$ is the stabilizer in $G$ of a finite subset $J\subset\Psi$. 
Then $G/U$ is identified with the set $\binom{\Psi}{\#J}$ of all 
subsets in $\Psi$ of order $\#J$. 
\item $G$ is the automorphism group of an algebraically closed extension 
$F$ of a field $k$, $U$ is the stabilizer in $G$ of an algebraically 
closed subextension $L|k$ of finite transcendence degree. Then $G/U$ is 
identified with the set of all subextensions in $F|k$ isomorphic to $L|k$. 
\end{enumerate} 

\begin{lemma} \label{no-simple-submod} Let $G$ be a group acting 
on a field $K$. Let $U\subset G$ be a subgroup such that an element 
$g\in G$ acts identically on $K^U$ if and only if $g\in U$. Then there 
are no irreducible $K$-semilinear subrepresentations in $K[G/U]$, 
unless $U$ is of finite index in $G$. If $G$ acts faithfully on $K$ and 
$U$ is of finite index in $G$ then $K[G/U]$ is trivial. \end{lemma} 

{\sc Example.} Let $G$ be a group acting on a field $K$; $U\subset G$ 
be a maximal proper subgroup. Assume that $K^U\neq K^G$. Then we are 
under assumptions of Lemma \ref{no-simple-submod}. 

More particularly, if $G=\Sy_{\Psi}$, $U=\Sy_{\Psi,I}$ (so 
$G/U\cong\binom{\Psi}{\#I}$) and $K^{\Sy_{\Psi,I}}\neq k:=
K^{\Sy_{\Psi}}$ for a finite subset $I\subset\Psi$ then there are no 
irreducible $K$-semilinear subrepresentations in $K[\binom{\Psi}{\#I}]$. 

\begin{proof} By Artin's independence of characters theorem (applied 
to the one-dimensional characters $g:(K^U)^{\times}\to K^{\times}$), 
the morphism $K[G/U]\to\prod_{(K^U)^{\times}}K$, given by 
$\sum_gb_g[g]\mapsto(\sum_gb_gf^g)_{f\in(K^U)^{\times}}$, is injective. 
Then, for any non-zero element $\alpha\in K[G/U]$, there exists 
an element $Q\in K^U$ such that the morphism $K[G/U]\to K$, given by 
$\sum_gb_g[g]\mapsto\sum_gb_gQ^g$, does not vanish on $\alpha$. 
Then $\alpha$ generates a subrepresentation $V$ surjecting onto $K$. 
If $V$ is irreducible then it is isomorphic to $K$, so $V^G\neq 0$. 
In particular, $K[G/U]^G\neq 0$, which can happen only if index of 
$U$ in $G$ is finite. 

If $[G:U]<\infty$ set $U'=\cap_{g\in G/U}gUg^{-1}$. This is a normal 
subgroup of finite index. Then $K[G/U']=K\otimes_{K^{U'}}K^{U'}[G/U']$ 
and $K^{U'}[G/U']\cong(K^{U'})^{[G:U']}$ is trivial by 
Hilbert's theorem 90, so we get $K[G/U']\cong K^{[G:U']}$. \end{proof} 

\begin{lemma} \label{restriction} Let $G$ be a permutation 
group, $A$ be an associative ring endowed with a smooth 
$G$-action and $U\subseteq G$ be an open subgroup. Then 
any smooth $A\langle G\rangle$-module is also smooth when 
considered as an $A\langle U\rangle$-module. Suppose that 
the set $U\backslash G/U'$ is finite for any open subgroup 
$U'\subseteq G$. Then the restriction of any smooth finitely generated 
$A\langle G\rangle$-module to $A\langle U\rangle$ 
is a finitely generated $A\langle U\rangle$-module. \end{lemma} 
\begin{proof} The $A\langle G\rangle$-modules $A[G/U']$ for all open 
subgroups $U'$ of $G$ form a generating family of the category of 
smooth $A\langle G\rangle$-modules. It suffices, thus, to check 
that $A[G/U']$ is a finitely generated $A\langle U\rangle$-module for all open 
subgroups $U'$ of $G$. Choose representatives $\alpha_i\in G/U'$ of 
the elements of $U\backslash G/U'$. Then $G/U'=\coprod_iU\alpha_i$, 
so $A[G/U']\cong\bigoplus_iA[U/(U\cap\alpha_iU'\alpha_i^{-1})]$ 
is a finitely generated $A\langle U\rangle$-module. \end{proof} 

{\sc Examples.} 1. The finiteness assumption of Lemma \ref{restriction} 
is valid for any open subgroup $G$ of a group of $\Sy$-type, 
as well as for any compact group $G$. 

2. The restriction functor splits the indecomposable generators 
into finite direct sums of indecomposable generators via canonical 
isomorphisms of $A\langle G_J\rangle$-modules $A[\binom{\Psi}{t}_q]
=\bigoplus_{\Lambda\subseteq J}M_{\Lambda}$, where $M_{\Lambda}$ is 
the free $A$-module on 
the set of all subobjects of $\Psi$ of length $t$ and meeting $J$ along $\Lambda$. 

\begin{lemma} \label{level-of-quotient} Let $s\ge 0$ be an integer 
and $M$ be a quotient of the $K\langle\Sy_{\Psi}\rangle$-module 
$K[\binom{\Psi}{s}]$ by a non-zero submodule $M_0$. 
Then there is a finite subset $I\subset\Psi$ such that the 
$K\langle\Sy_{\Psi|I}\rangle$-module $M$ is isomorphic to a 
quotient of $\bigoplus_{j=0}^{s-1}K[\binom{\Psi\smallsetminus I}{j}]
^{\oplus\binom{\# I}{s-j}}$. \end{lemma} 
\begin{proof} Let $\alpha=\sum_{S\subseteq J}a_S[S]\in M_0$ be 
a non-zero element for a finite set $J\subset\Psi$. Fix some 
$S\subseteq J$ with $a_S\neq 0$. Set $I:=J\smallsetminus S$. 
Then the morphism of $K\langle\Sy_{\Psi|I}\rangle$-modules 
$K\langle\Sy_{\Psi|I}\rangle\alpha\oplus
\bigoplus_{\varnothing\neq\Lambda\subseteq I}
K[\binom{\Psi\smallsetminus I}{s-\#\Lambda}]\to K[\binom{\Psi}{s}]$, 
given (i) by the inclusion on the first summand and (ii) 
by $[T]\mapsto[T\cup\Lambda]$ on the summand corresponding 
to $\Lambda$, is surjective. \end{proof} 

In the following result, our principal examples of the ring $A$ will be 
division rings endowed with an $\Sy_{\Psi}$-action, though localization of 
$\mathbb Z[x~|~x\in\Psi]$ at all non-constant indecomposable polynomials 
gives one more example. 
\begin{proposition} \label{noetherian-generators} Let $A$ be an associative 
left noetherian ring endowed with an arbitrary $\Sy_{\Psi}$-action. 
Then the left $A\langle U\rangle$-module $A[\Psi^s]$ is noetherian for 
any integer $s\ge 0$ and any open subgroup $U\subseteq\Sy_{\Psi}$. 
If the $\Sy_{\Psi}$-action on $A$ is smooth then any smooth finitely generated 
$A\langle\Sy_{\Psi}\rangle$-module is noetherian. \end{proposition} 
\begin{proof} We need to show that any $A\langle U\rangle$-submodule 
$M\subset A[\Psi^s]$ is finitely generated for all $U=\Sy_{\Psi|S}$ with finite 
$S\subset\Psi$. We proceed by induction on $s\ge 0$, the case $s=0$ 
being trivial. Assume that $s>0$ and the $A\langle U\rangle$-modules 
$A[\Psi^j]$ are noetherian for all $j<s$. Fix a subset 
$I_0\subset\Psi\smallsetminus S$ of order $s$. 

Let $M_0$ be the image of $M$ under the $A$-linear projector 
$\pi_0:A[\Psi^s]\to A[I_0^s]\subset A[\Psi^s]$ omitting all 
$s$-tuples containing elements other than those of $I_0$. As $A$ 
is noetherian and $I_0^s$ is finite, the $A$-module $M_0$ is 
finitely generated. Let the $A$-module $M_0$ be generated by the images of 
some elements $\alpha_1,\dots,\alpha_N\in M\subseteq A[\Psi^s]$. 
Then $\alpha_1,\dots,\alpha_N$ belong to the $A$-submodule 
$A[I^s]$ of $A[\Psi^s]$ for some finite subset $I\subset\Psi$. 

Let $J\subset I\cup S$ be the complement to $I_0$. For 
each pair $\gamma=(j,x)$, where $1\le j\le s$ and $x\in J$, set 
$\Psi^s_{\gamma}:=\{(x_1,\dots,x_s)\in\Psi^s~|~x_j=x\}$. This is 
a smooth $\Sy_{\Psi|J}$-set. Then the set $\Psi^s$ is the union 
of the $\Sy_{\Psi|J}$-orbit consisting of $s$-tuples of pairwise 
distinct elements of $\Psi\smallsetminus J$ and of a finite union 
of $\Sy_{\Psi|J}$-orbits embeddable into $\Psi^{s-1}$: 
$\bigcup_{\gamma}\Psi^s_{\gamma}\cup\bigcup_{1\le i<j\le s}\Delta_{ij}$, 
where $\Delta_{ij}:=\{(x_1,\dots,x_s)\in\Psi^s~|~x_i=x_j\}$ are diagonals. 

As (i) $M_0\subseteq\sum_{j=1}^NA\alpha_j+
\sum_{\gamma\in\{1,\dots,s\}\times J}A[\Psi^s_{\gamma}]$, (ii) 
$g(M_0)\subset A[\Psi^s]$ is determined by $g(I_0)$, (iii) for any 
$g\in U$ such that $g(I_0)\cap J=\varnothing$ there exists 
$g'\in U_J$ with $g(I_0)=g'(I_0)$ ($U_J$ acts transitively 
on the $s$-configurations in $\Psi\smallsetminus J$), 
one has inclusions of $A\langle U_J\rangle$-modules 
\[\sum_{j=1}^NA\langle U\rangle\alpha_j\subseteq M\subseteq\sum_{g\in U}
g(M_0)\subseteq\sum_{g\in U_J}g(M_0)
+\sum_{\gamma\in\{1,\dots,s\}\times J}A[\Psi^s_{\gamma}].\] 
On the other hand, $g(M_0)\subseteq g(\sum_{j=1}^NA\alpha_j)+
\sum_{\gamma\in\{1,\dots,s\}\times J}A[\Psi^s_{\gamma}]$ 
for $g\in U_J$, 
and therefore, the $A\langle U_J\rangle$-module 
$M/\sum_{j=1}^NA\langle U\rangle\alpha_j$ becomes a subquotient of the 
noetherian, by the induction assumption, $A\langle U_J\rangle$-module 
$\sum_{\gamma\in\{1,\dots,s\}\times J}A[\Psi^s_{\gamma}]$, so 
the $A\langle U_J\rangle$-module $M/\sum_{j=1}^NA\langle U\rangle\alpha_j$ 
is finitely generated, and thus, $M$ is finitely generated as well. \end{proof} 

\begin{corollary} \label{noetherian} Let $A$ be a left noetherian 
associative ring endowed with a smooth $\Sy_{\Psi}$-action. Then 
{\rm (i)} any smooth finitely generated left $A\langle\Sy_{\Psi}\rangle$-module 
$W$ is noetherian if considered as a left 
$A\langle U\rangle$-module for any open subgroup $U\subseteq\Sy_{\Psi}$; 
{\rm (ii)} the category of smooth $A\langle\Sy_{\Psi}\rangle$-modules 
is locally noetherian, i.e., any smooth finitely generated left 
$A\langle\Sy_{\Psi}\rangle$-module is noetherian. \end{corollary} 
\begin{proof} The module $W$ is a quotient of a finite direct 
sum of $A[\Psi^m]$ for some integer $m\ge 0$, while $A[\Psi^m]$ 
are noetherian by Proposition \ref{noetherian-generators}. 
\end{proof} 

The next result generalizes a description of representations 
$k[\binom{\Psi}{s}]$ of $\Sy_{\Psi}$ from \cite{CaminaEvans}. 
\begin{lemma} \label{length-binom} In notation preceeding Corollary 
\ref{def-binom}, let $q\ge 1$ be either 1 or a primary number, $S$ 
be a finite or infinite set (if $q=1$) or an $\mathbb F_q$-vector 
space, $G$ be the automorphism group of $S$. Let the morphism of 
$G$-modules $\partial_s^S:{\mathbb Z}[\binom{S}{s}_q]\to
{\mathbb Z}[\binom{S}{s-1}_q]$ be defined by 
$[T]\mapsto\sum_{T'\subset T}[T']$. Then {\rm (i)} 
$\partial_{s+1}^S\otimes\mathbb Q$ is surjective if and only if length 
of $S$ is $>2s$, {\rm (ii)} the representation $\mathbb Q[\binom{S}{s}_q]$ 
of $G$ is of length $\min(s,\mathrm{length}(S)-s)+1$ and 
{\rm (iii)} the irreducible quotients of its composition series are 
absolutely irreducible and pairwise non-isomorphic. \end{lemma} 
\begin{proof} Fix a subobject $T$ of $S$ of length $s$. Then any element 
$\varphi\in\Hom_G({\mathbb Z}[\binom{S}{s}_q],{\mathbb Z}[\binom{S}{t}_q])$ 
is determined by the image of $T$. As $\varphi(T)$ is fixed by the 
stabilizer of $T$ and the latter acts transitively on the set of all 
subobjects of $S$ of length $t$ and with a given length of intersection 
with $T$, one has $\varphi([T])=
\sum_{i=0}^{\min(s,t)}a_i\sum_{\mathrm{length}(T'\cap T)=i}[T']$ 
for a collection of coefficients $a_i\in\mathbb Q$ if $S$ is finite 
and $\varphi([T])=a\sum_{T'\subseteq T}[T']$ for a coefficient 
$a\in\mathbb Q$ if $S$ is infinite. 

Assume first that $S$ is finite. Comparing the cases 
$t=s$ and $t=s-1$ and arguing by induction, we see that (i) 
${\mathbb Q}[\binom{S}{s}_q]$ is a direct sum of $s+1$ pairwise 
non-isomorphic absolutely irreducible subrepresentation, (ii) 
${\mathbb Q}[\binom{S}{s}_q]$ embeds into ${\mathbb Q}[\binom{S}{s+1}_q]$. 

Assume that length of $S$ is $>2s$. The morphisms 
$\partial^S_s\otimes\mathbb Q,\partial_{s-1}^S\otimes\mathbb Q,
\dots,\partial^S_1\otimes\mathbb Q$ are non-injective, but they 
cannot drop length of modules by more than 1, since their 
composition 
$\partial^S_s\partial_{s-1}^S\cdots\partial^S_1$ is non-zero. 
This means that their kernels are irreducible, and thus, 
they are surjective. 

Now assume that $S$ is infinite. Clearly, the surjectivity of 
$\partial^S_s\otimes\mathbb Q$ and the irreducibility of its kernel 
follows from the case of sufficiently large finite $S$. By Lemma 
\ref{length-upper-bound}, length of ${\mathbb Q}[\binom{S}{s}_q]$ 
does not exceed $s+1$, so it is precisely $s+1$. 

If length of $S$ is $\le 2s$ then $\partial^S_{s+1}\otimes{\mathbb Q}$ 
is not surjective, since then $\dim_{{\mathbb Q}}{\mathbb Q}[\binom{S}{s+1}_q]=
\#\binom{S}{s+1}_q<\#\binom{S}{s}_q=
\dim_{{\mathbb Q}}{\mathbb Q}[\binom{S}{s}_q]$. 

There are no other irreducible subrepresentation if $S$ is infinite, 
since ${\mathbb Q}[\binom{S}{s}_q]$ is injective and indecomposable, 
cf. \cite[Proposition 6.9]{BucurDeleanu}. \end{proof} 

\begin{corollary} In notation of Lemma \ref{length-binom}, let $A$ 
be a torsion-free commutative integral domain (or a division ring) 
endowed with the trivial $G$-action. Then, any $G$-submodule 
$M\subseteq A[\binom{S}{s}_q]$ with 
$(M\otimes\mathrm{Frac}(A))\cap A[\binom{S}{s}_q]=M$ is 
the kernel of $\partial^S_s\partial_{s-1}^S\cdots\partial_{s-i}^S$ 
for some $i$. In particular, the $A[G]$-module $A[\binom{S}{s}_q]$ 
is of length $s+1$ and $\partial_s^S\otimes A$ is surjective if $A$ 
is a division ring and length of $S$ is $>2s$. \qed \end{corollary} 

\begin{corollary} Let $A$ be a left noetherian associative ring endowed 
with a smooth $\Sy_{\Psi}$-action. Then {\rm (i)} any sum of smooth 
injective left $A\langle\Sy_{\Psi}\rangle$-modules 
is again injective; {\rm (ii)} any smooth injective left 
$A\langle\Sy_{\Psi}\rangle$-module is a sum of uniquely determined (upto 
non-unique isomorphism) collection of indecomposable smooth injective 
left $A\langle\Sy_{\Psi}\rangle$-modules; {\rm (iii)} injective 
hull of a smooth noetherian left $A\langle\Sy_{\Psi}\rangle$-module 
is a finite sum of indecomposables. \end{corollary} 
\begin{proof} (i) is \cite[Corollary 6.50]{BucurDeleanu} by Corollary 
\ref{noetherian}; (ii) is \cite[Proposition 6.51]{BucurDeleanu}, 
again by Corollary \ref{noetherian}; 
(iii) is \cite[Proposition 6.41]{BucurDeleanu}. \end{proof} 

For any $N,M\ge 0$, denote by $S_{M,N}$ the (purely transcendental over $k$ of 
transcendence degree $N$) field of rational functions in $N$ variables over 
a field $k$ symmetric both in the first $M$ and in the remaining $N-M$ variables. 
For instance, (i) $S_{M,N}:=0$ if $N<M$, (ii) $S_{N,N}$ is the field of 
symmetric rational functions in $N$ variables over $k$; (iii) $S_{N,N+1}$ is 
the field of rational functions in the $(N+1)$-st variable over the field 
$S_{N,N}$ of symmetric rational functions in the first $N$ variables. 

For any $Q\in S_{M,N}$ and any $J\subseteq I\subset\Psi$ with 
$\# I=N$ and $\# J=M$, denote by $Q(J\subseteq I)\in k(\Psi)$ the value 
of $Q$ on the collection $I$, where the first $M$ arguments are in $J$. 
\begin{lemma} \label{morphismes-entre-generat} Let $K=k(\Psi)$ for a field $k$. 
Then there is a canonical isomorphism of $k$-vector spaces 
\[\Hom_{K\langle\Sy_{\Psi}\rangle}(K[\binom{\Psi}{N}],K[\binom{\Psi}{M}])
=\left\{\begin{array}{ll} S_{M,N},&\text{{\rm if $N\ge M$,}}\\
0,&\text{{\rm if $N<M$,}} \end{array}\right.\qquad Q:[T]\mapsto
\sum_{J\subseteq T}Q(J\subseteq T)[J].\]
Under this isomorphism, the  
composition $K[\binom{\Psi}{N}]\stackrel{R}{\longrightarrow}
K[\binom{\Psi}{M}]\stackrel{Q}{\longrightarrow}K[\binom{\Psi}{L}]$ 
is given by \[(Q\circ R)(J_0\subseteq T):=\sum_{J_0\subseteq 
J\subseteq T}Q(J\subseteq T)R(J_0\subseteq J),\quad\mbox{where length 
of $J$ is $M$.}\] \end{lemma} 
\begin{proof} Fix a subset $I\subset\Psi$ of cardinality $N$ and a subset 
$J_0\subset I$ of cardinality $M$. Then the morphisms $K[\binom{\Psi}{N}]
\to K[\binom{\Psi}{M}]$ are in one-to-one correspondence with the elements 
of $K[\binom{\Psi}{M}]^{\Sy_{\Psi,I}}$: to a morphism we associate its 
value on $I$; any morphism is determined by its value on $I$. Any element 
in $K[\binom{\Psi}{M}]^{\Sy_{\Psi,I}}$ is of the form 
$\sum_{J\subseteq I}Q_J[J]$, where $\sigma Q_J=Q_{\sigma(J)}$ for any 
$\sigma\in\Sy_I$. In particular, $Q_J$ for all $J$ are determined by 
$Q_{J_0}$ and $Q_{J_0}$ is a rational function over $k$ in variables 
$I$ symmetric in the variables $J_0$ and in the variables 
$I\smallsetminus J_0$, i.e., the space of morphisms is identified 
with $S_{M,N}$. \end{proof} 

\section{Triviality of the smooth finite-dimensional semilinear 
representations of $\Sy_{\Psi}$} The following result is analogous 
to the case of general linear group of \cite[Proposition 5.4]{repr}. 
\begin{lemma} \label{triviality-finite-dim} Let $K=k(\Psi)$ 
for a field $k$. Then 
any smooth finite-dimensional $K$-semilinear 
representation of $\Sy_{\Psi}$ is trivial. \end{lemma} 
\begin{proof} Set $G:=\Sy_{\Psi}$ and let $b\subset V$ be a $K$-basis, 
pointwise fixed by an open subgroup of $G$, so $b\subset V_I:=V^{G_I}$ 
for a finite subset $I\subset\Psi$. By Lemma~\ref{inject} (with 
$\chi\equiv 1$), the multiplication maps $V_I\otimes_{K_I}K=
(V_I\otimes_{K_I}K_J)\otimes_{K_J}K\to V_J\otimes_{K_J}K\to V$ 
are injective for any subset $J\subseteq\Psi$ containing $I$, where 
$K_J:=K^{G_J}$. The composition is an isomorphism, so 
$V_I\otimes_{K_I}K_J\to V_J$ is an isomorphism as well. 
In particular, $f_{\sigma}=id_V$ if $\sigma\in G_I$, where 
$(f_{\sigma}\in\GL_K(V))_{\sigma}$ is the 1-cocycle of 
the $G$-action in the basis $b$. Clearly, (i) $f_{\sigma}$ 
depends only on the class $\sigma|_I$ of $\sigma$ in 
$G/G_I=\{\text{embeddings of $I$ into $\Psi$}\}$, (ii) 
$f_{\sigma}\in\GL_{K_{I\cup\sigma(I)}}(V_{I\cup\sigma(I)})$. 

Assume that $I,\sigma(I),\tau\sigma(I)$ are disjoint, $X,Y,Z$ are 
the standard collections of the elementary symmetric functions in 
$I,\tau(I),\tau\sigma(I)$, respectively. Then the cocycle condition 
$f_{\tau\sigma}=f_{\tau}f_{\sigma}^{\tau}$ 
(where $f_{\sigma}^{\tau}\in\GL_{K_{\tau(I)\cup\tau\sigma(I)}}
(V_{\tau(I)\cup\tau\sigma(I)})$) becomes $\Phi(X,Z)=\Phi(X,Y)\Phi(Y,Z)$ 
and $\Phi(Y,X)=\Phi(X,Y)^{-1}$, where $f_{\tau\sigma}=\Phi(X,Z)$, etc. 
If $k$ is infinite then there is a $k$-point $Y_0$, where $\Phi(X,Y)$ 
and $\Phi(Y,Z)$ are regular. If $k$ is finite then there is a finite field 
extension $k'|k$ and a $k'$-point $Y_0$, where $\Phi(X,Y)$ and $\Phi(Y,Z)$ 
are regular. Specializing $Y$ to such $Y_0$, we get $\Phi(X,Z)=\Phi(X,Y_0)
\Phi(Y_0,Z)=\Phi(X,Y_0)\Phi(Z,Y_0)^{-1}$. Then $\Phi(X,Y_0)$ transforms 
$b$ to a basis $b'$ fixed by all $\sigma\in G$ such that $\sigma(I)$ 
does not meet $I$. As the elements $\sigma\in G$ such that $\sigma(I)$ does 
not meet $I$ generate the whole group $G$, the basis $b'$ is fixed by entire 
$G$. This gives an embedding of $V$ into a (finite) direct sum of copies of 
$K\otimes_kk'$, which is itself a (finite) direct sum of copies of $K$, 
and finally, so is $V$ as well. \end{proof} 

\section{Structure of $K[\Psi]$}
\begin{lemma} \label{some-Exts} Let $K$ be a field endowed with a 
smooth $\Sy_{\Psi}$-action. For any pair of integers $m,n\ge 0$ one has 
$\Ext^1(K[\Psi^m],K[\Psi^n])=0$ in the category of smooth $K$-semilinear 
representations of $\Sy_{\Psi}$. In particular, the restriction morphism 
\[\rho_{m,n}:\Hom_{K\langle\Sy_{\Psi}\rangle}(K[\binom{\Psi}{m}],
K[\binom{\Psi}{n}])\to\Hom_{K\langle\Sy_{\Psi}\rangle}
(V,K[\binom{\Psi}{n}])\] is surjective if $V=K[\binom{\Psi}{m}]^{\circ}$ 
or if $m=1$, $K=k(\Psi)$ and $0\neq V\subseteq K[\Psi]$ 
(and injective if $n>0$). \end{lemma} 
\begin{proof} Consider first the case of $m=0$. Let 
$0\to K[\Psi^n]\to E\to K\to 0$ be an extension. Choose a section 
$v\in E$ (projecting to $1\in K$). Then $v\in E^{\Sy_{\Psi|I}}$ 
for a finite subset $I\subset\Psi$, and therefore, 
$f_{\sigma}:=\sigma v-v\in(K[\Psi^n])^{\Sy_{\Psi|I\cup\sigma(I)}}
\subseteq K[(I\cup\sigma(I))^n]$ for any $\sigma\in\Sy_{\Psi}$. Then 
$f_{\sigma}=\sum_{\xi\in(I\cup\sigma(I))^n}a_{\sigma,\xi}[\xi]$. 
Let $\sigma,\tau\in\Sy_{\Psi}$ be such elements 
that $\#(I\cup\sigma(I)\cup\tau\sigma(I))=3\# I$. Then $f_{\sigma}^{\tau}
=\sum_{\xi\in(\tau(I)\cup\tau\sigma(I))^n}a_{\sigma,\tau^{-1}\xi}
^{\tau}[\xi]$ and $f_{\tau}=\sum_{\xi\in(I\cup\tau(I))^n}
a_{\tau,\xi}[\xi]$. From the 1-cocycle condition 
$f_{\tau\sigma}=f_{\tau}+f_{\sigma}^{\tau}$ we see that $a_{\tau,\xi}=0$, 
unless support of $\xi$ is contained either in $I$ or in $\tau(I)$. 
Moreover, $a_{\tau,\xi}+a_{\sigma,\tau^{-1}\xi}^{\tau}=0$ 
if support of $\xi$ is contained in $\tau(I)$; 
$a_{\tau\sigma,\xi}=a_{\sigma,\xi}$ if support of $\xi$ is contained in 
$I$; $a_{\tau\sigma,\xi}=a_{\sigma,\tau^{-1}\xi}^{\tau}$ 
if support of $\xi$ is contained in $\tau\sigma(I)$. 

Then the element $v':=v-\sum_{\xi\in I^n}a_{\sigma,\xi}[\xi]$ is fixed 
by all elements $\sigma\in\Sy_{\Psi}$ such that $\#(I\cup\sigma(I))=2\# I$. 
As such $\sigma$ generate the group $\Sy_{\Psi}$, 
we get $v'\in E^{\Sy_{\Psi}}$, i.e., our extension is split. 

Now let $m$ be arbitrary. We split $\Psi^m$ into disjoint union of 
$\Sy_{\Psi}$-orbits $O$. For a fixed $O$, let $0\to K[\Psi^n]\to E
\stackrel{p}{\longrightarrow}K[O]\to 0$ be a smooth extension of 
$K\langle\Sy_{\Psi}\rangle$-modules. We know already that 
$0\to K[\Psi^n]\to p^{-1}(K\cdot[\xi])\stackrel{p}{\longrightarrow}K\cdot
[\xi]\to 0$ is a split extension of $K\langle\Sy_{\Psi|J}\rangle$-modules 
for any $\xi\in O$, where $J\subset\Psi$ is the support of $\xi$. Then we 
can choose $v\in E^{\Sy_{\Psi|J}}$ with $p(v)=[\xi]$. Such $v$ spans in 
$E$ a $K\langle\Sy_{\Psi}\rangle$-submodule identified by $p$ with 
$K[O]$, i.e., the extension splits. Then the short exact sequence 
$0\to V\to K[\binom{\Psi}{m}]\to K^N\to 0$ for some integer $N\ge 0$
induces the surjection $\rho_{m,n}$ with the kernel 
$\Hom_{K\langle\Sy_{\Psi}\rangle}(K,K[\binom{\Psi}{n}])$, which is 
$K^{\Sy_{\Psi}}$ for $n=0$ and $0$ for $n>0$ 
(cf. Lemma \ref{no-simple-submod}). \end{proof} 

\begin{lemma} \label{weight-1-quotients} 
\label{non-zero-weight-1} The cokernel of any non-zero morphism of 
$K\langle\Sy_{\Psi}\rangle$-modules $\varphi:K[\binom{\Psi}{s}]\to K[\Psi]$ 
is at most $(s-1)$-dimensional. In particular, any non-zero 
submodule of $K[\Psi]$ is of finite codimension. \end{lemma} 
\begin{proof} Fix a subset $\{b_1,\dots,b_{s-1}\}$ 
of $\Psi$ of order $(s-1)$. Then, for any element $b\in\Psi$, $b\neq b_i$, 
the set $\{b,b_1,\dots,b_{s-1}\}$ is sent by $\varphi$ to a linear 
combination of $b,b_1,\dots,b_{s-1}$ with non-zero coefficients, 
and therefore, the image of $b$ in the cokernel of $\varphi$ is 
a linear combination of images of $b_1,\dots,b_{s-1}$. \end{proof} 

\begin{lemma} Let $K=k(\Psi)$ be endowed with the standard 
$\Sy_{\Psi}$-action. \begin{enumerate} \item There are natural bijections 
\begin{enumerate} \item between the $K\langle\Sy_{\Psi}\rangle$-submodules 
of $K[\Psi]$ of codimension $s$ and the $s$-dimensional $k$-vector 
subspaces in $k(T)$; \item between the isomorphism classes of 
$K\langle\Sy_{\Psi}\rangle$-submodules of $K[\Psi]$ of codimension 
$s$ and the $k(T)^{\times}$-orbits of the $s$-dimensional $k$-vector 
subspaces in the space $k(T)$ of rational functions in one variable $T$. 
\end{enumerate} \item An element $\sum_{t}q_t[t]\in K[\Psi]$ 
is a generator if and only if $\sum_{t}q_tQ(t)\neq 0$ 
for any $Q\in k(T)^{\times}$. \end{enumerate} \end{lemma} 
{\it Proof.} \begin{enumerate} \item By Lemma \ref{non-zero-weight-1}, 
the $K$-vector space $K[\Psi]/M$ is finite-dimensional for any 
non-zero $K\langle\Sy_{\Psi}\rangle$-submodule $M$ of $K[\Psi]$, 
so by Lemma \ref{triviality-finite-dim} the 
$K\langle\Sy_{\Psi}\rangle$-module $K[\Psi]/M$ is isomorphic to 
a sum of copies of $K$, and therefore, 
$M$ is the common kernel of the elements of a finite-dimensional 
$k$-vector subspace of $\Hom_{K\langle\Sy_{\Psi}\rangle}(K[\Psi],K)$ 
(identified with the field of rational functions $k(T)$). Clearly, 
the $\End_{K\langle\Sy_{\Psi}\rangle}(K[\Psi])^{\times}$-action 
preserves the isomorphism classes of the common kernels (since 
$\End_{K\langle\Sy_{\Psi}\rangle}(K[\Psi])=k(T)$ is a field). On the 
other hand, $\Hom_{K\langle\Sy_{\Psi}\rangle}(K[\Psi]/M,K[\Psi])=0$ 
and (by Lemma \ref{some-Exts}) 
$\Ext^1_{K\langle\Sy_{\Psi}\rangle}(K[\Psi]/M,K[\Psi])=0$, 
so the restriction morphism $\End_{K\langle\Sy_{\Psi}\rangle}(K[\Psi])
\to\Hom_{K\langle\Sy_{\Psi}\rangle}(M,K[\Psi])$ is an isomorphism, i.e., 
any morphism between $K\langle\Sy_{\Psi}\rangle$-submodule $M$ of 
$K[\Psi]$ is induced by an endomorphism of $K[\Psi]$ (identified 
with an element of $k(T)$). 

\item Any $Q\in k(T)^{\times}$ such that 
$\sum_{t}q_tQ(t)=0$ determines a non-zero morphism $K[\Psi]\to K$ trivial 
on the submodule generated by $\alpha:=\sum_{t}q_t[t]$, so $\alpha$ is 
not a generator. If $\alpha$ generate a proper submodule $M\subset K[\Psi]$ 
then the quotient $K[\Psi]/M$ is finite-dimensional, so by 
Lemma~\ref{triviality-finite-dim} it admits a quotient isomorphic to $K$. 
Finally, any morphism $K[\Psi]\to K$ is given by some $Q\in k(T)^{\times}$. 
\qed \end{enumerate} 

\begin{lemma} 
Let $K=k(\Psi)$ be endowed with the standard $\Sy_{\Psi}$-action. 
Then the following conditions on a non-zero rational function $q(X,Y)$ 
over $k$ are equivalent \begin{enumerate} \item \label{not-gen} 
$q:K[\binom{\Psi}{2}]\to K[\Psi]$, $[\{a,b\}]\mapsto q(a,b)[a]+q(b,a)[b]$, 
is not surjective (in other words, $q(a,b)[a]+q(b,a)[b]$ is not a generator of $K[\Psi]$), 
\item \label{cokern-nontr} the cokernel of $q:K[\binom{\Psi}{2}]\to K[\Psi]$ 
is isomorphic to $K$, \item \label{non-0-non-surj} 
$q(X,Y)=(X-Y)S(Y)R(X,Y)$ for some $S$ and a symmetric $R$, 
\item \label{exists-annul} there exists some $S(X)\neq 0$ such that 
$q(a,b)S(a)+q(b,a)S(b)=0$, \item \label{exists-exact} there exists some 
$S(X)$ such that the sequence 
$K[\binom{\Psi}{2}]\stackrel{q}{\longrightarrow}K[\Psi]
\stackrel{S}{\longrightarrow}K$ is exact. \end{enumerate}\end{lemma} 
\begin{proof} (\ref{cokern-nontr})$\Rightarrow$(\ref{not-gen}) and 
(\ref{exists-exact})$\Rightarrow$(\ref{exists-annul}) are trivial; 
(\ref{exists-annul})$\Rightarrow$(\ref{non-0-non-surj})$\Rightarrow$%
(\ref{cokern-nontr}) are evident. 

(\ref{not-gen})$\Rightarrow$(\ref{exists-exact}). Let $A=\{a_0,\dots,a_s\}$ 
be a subset of $\Psi$ of order $s+1$. Then $A\smallsetminus\{a_j\}$ is 
sent by $q$ to $\sum_{i\neq j}q(A\smallsetminus\{a_i,a_j\};a_i)\{a_i\}$, 
so the image contains $s+1$ elements 
$\sum_{i\neq j}q(A\smallsetminus\{a_i,a_j\};a_i)\{a_i\}$ for 
$0\le j\le s$. In the case $s=2$ these three elements span a vector 
space of dimension 2 or 3. The dimension is 2 if and only if $D:=\det
\left(\begin{matrix}q(a_0,a_1)&q(a_1,a_0)&0\\ 0&q(a_1,a_2)&q(a_2,a_1)\\ q(a_0,a_2)&0&q(a_2,a_0)
\end{matrix}\right)=q(a_0,a_1)q(a_1,a_2)q(a_2,a_0)+q(a_1,a_0)q(a_2,a_1)q(a_0,a_2)$ vanishes. 
Let $q(X,Y)=P(X)S(Y)\prod_i\Phi_i(X,Y)^{m_i}$ be a decomposition 
into a product of irreducibles. Then 
\[D=R(a_0)R(a_1)R(a_2)(\prod_i(\Phi_i(a_0,a_1)\Phi_i(a_1,a_2)\Phi_i(a_2,a_0))^{m_i}
+\prod_i(\Phi_i(a_1,a_0)\Phi_i(a_2,a_1)\Phi_i(a_0,a_2))^{m_i}),\] where $R=PS$, 
vanishes if and only if $\prod_i(\Phi_i(a_0,a_1)\Phi_i(a_1,a_2)
\Phi_i(a_2,a_0))^{m_i}$ is a skew symmetric function, i.e., 
$\prod_i\Phi_i(X,Y)^{m_i}$ is a skew symmetric function. \end{proof} 

\section{Cyclicity of the smooth finitely generated semilinear 
representations of $\Sy_{\Psi}$} 
The following result extends the existence of a cyclic vector in 
any finite-degree non-degenerate semilinear representation of an 
endomorphism of infinite order, cf., e.g., \cite[Lemma 2.1]{pgl}. 

\begin{lemma} \label{surjectivity} Let $G$ be a permutation group, 
$K$ be a field endowed with a smooth $G$-action such that any open 
subgroup of $G$ contains an element inducing on $K$ an automorphism 
of infinite order. 
Then any smooth finitely generated $K\langle G\rangle$-module $W$ admits 
a cyclic vector. \end{lemma} 
\begin{proof} A finite system $S$ of generators of the 
$K\langle G\rangle$-module $W$ is fixed by an open subgroup 
$U\subseteq G$. By \cite[Lemma 2.1]{pgl}, the 
$K\langle U\rangle$-module spanned by $S$ admits a cyclic vector $v$. 
Then $v$ is a cyclic vector of the $K\langle G\rangle$-module $W$. 
\end{proof} 

\begin{corollary} \label{presentation-simple-semilin} 
Let $K$ be a field endowed with a smooth faithful $\Sy_{\Psi}$-action. 
Then any smooth simple left $K\langle\Sy_{\Psi}\rangle$-module 
is isomorphic to $K[\binom{\Psi}{s}]/K\langle\Sy_{\Psi}\rangle\alpha$ 
for some $\alpha\in K[\binom{\Psi}{s}]$. \end{corollary} 
\begin{proof} By Lemma \ref{semilin-gener}, any smooth 
simple left $K\langle\Sy_{\Psi}\rangle$-module is isomorphic 
to a quotient of $K[\binom{\Psi}{s}]$ for an appopriate $s\ge 0$ by 
a left $K\langle\Sy_{\Psi}\rangle$-submodule $V$. By Proposition 
\ref{noetherian-generators}, the $K\langle\Sy_{\Psi}\rangle$-module $V$ 
is finitely generated, and thus, by Lemma \ref{surjectivity}, it is cyclic, i.e., 
it is generated by some $\alpha\in K[\binom{\Psi}{s}]$. \end{proof} 

\begin{corollary} Let $K=k(\Psi)$ for a field $k$. Then {\rm (1)} the $K$-semilinear 
representations of $\Sy_{\Psi}$ of the following 4 classes, where $s\ge 0$ 
is integer, are indecomposable: {\rm (i)}$_s$ $K[\binom{\Psi}{s}]$, where 
$s\neq 1$, {\rm (ii)} the $K$-semilinear subrepresentations of $K[\Psi]$, 
{\rm (iii)}$_s$ $K[\binom{\Psi}{s}]^{\circ}$, where $s\ge 2$; {\rm (2)} a pair of 
such representations consists of isomorphic ones only if they belong to the 
same class (i.e., to one of {\rm (i)}$_s$, {\rm (ii)}, {\rm (iii)}$_s$ 
for some $s$) and, in the case {\rm (ii)}, have the same codimension 
in $K[\Psi]$. \end{corollary} 
\begin{proof} This follows from $\End_{K\langle\Sy_{\Psi}\rangle}
(K[\binom{\Psi}{s}]^{\circ})=k$ and Lemma \ref{morphismes-entre-generat}. 

The short exact sequence $0\to V\to K[\Psi]\to K[\Psi]/V\to 0$ gives an 
exact sequence $0\to\Hom(K,K[\Psi]/V)\to\Ext^1(K,V)\to\Ext^1(K,K[\Psi])$. 
By Lemma \ref{some-Exts}, $\Ext^1(K,K[\Psi])=0$. By Lemma 
\ref{non-zero-weight-1}, $\dim_K(K[\Psi]/V)$ is finite if $V\neq 0$, 
so by Lemma \ref{triviality-finite-dim}, 
$\dim_k\Hom(K,K[\Psi]/V)=\dim_K(K[\Psi]/V)$ if $V\neq 0$. 
This implies that codimension of $V\neq 0$ in $K[\Psi]$ is 
$\dim_k\Ext^1(K,V)$. \end{proof} 

\begin{conjecture} \label{indec-injectives} Let $K=k(\Psi)$ for a field 
$k$ be endowed with the standard $\Sy_{\Psi}$-action. Then for any 
$s\ge 0$ the indecomposable smooth $K\langle\Sy_{\Psi}\rangle$-module 
$K[\binom{\Psi}{s}]$ is injective. Any indecomposable injective 
smooth $K\langle\Sy_{\Psi}\rangle$-module is isomorphic 
to $K[\binom{\Psi}{s}]$ for some $s\ge 0$. \end{conjecture} 

\appendix
\section{Smooth semilinear representations 
of groups exhausted by compact subgroups} \label{exhaust}
Let $K$ be a field and $\psi:G\hookrightarrow\Aut_{\mathrm{field}}(K)$ 
be a group of field automorphisms of $K$. Set $k:=K^G$. 

There is a bijection between $k$-lattices $U$ in a $K$-vector 
space $V$ and the trivial $K$-semilinear $G$-actions on $V$: 
\begin{gather*}\{\text{$k$-vector subspaces $U$ in $V$ such that 
$U\otimes_kK\to V$ is bijective}\}\\ 
\xi\downarrow\phantom{H^0(G,-)}\uparrow H^0(G,-)\\ 
\{\text{structures on $V$ of $K$-semilinear representation 
isomorphic to a direct sum of copies of $K$}\}.\end{gather*}

The set of isomorphism classes of non-degenerate $K$-semilinear 
$G$-actions on a $K$-vector space $V:=U\otimes_kK$ is canonically 
identified with the set $H^1(G,\GL_K(V))$. Namely, there 
is a unique $K$-semilinear $G$-action $\xi$ on $V$ identical on the 
$k$-lattice $U$. This gives rise to the $K$-semilinear $G$-action 
on $\End_K(V)$ by $f^{\tau}:=\xi(\tau)f\xi(\tau)^{-1}\in\End_K(V)$ 
if $f\in\End_K(V)$, so the matrix of $f^{\tau}$ in the basis 
$\tau(b)$ is the result of applying $\tau$ to the matrix of $f$ in 
a basis $b$ of $U$. For each element $\sigma\in G$ there exists a unique 
element $f_{\sigma}\in\GL_K(V)$ such that $\sigma|_U=f_{\sigma}|_U$, 
so $\sigma v=f_{\sigma}\xi(\sigma)v$. This implies $\tau\sigma v=
f_{\tau}\xi(\tau)f_{\sigma}\xi(\sigma)v=f_{\tau}f_{\sigma}^{\tau}
\xi(\tau\sigma)v=f_{\tau\sigma}\xi(\tau\sigma)v$ for all $v\in V$, 
so $f_{\tau\sigma}=f_{\tau}f_{\sigma}^{\tau}$ for all 
$\sigma,\tau\in G$. 

Suppose that $G$ is exhausted by its compact subgroups, i.e., any 
compact subset of $G$ is contained in a compact subgroup. 

Let $V$ be an $K$-vector space. By Theorem \ref{Satz90}, restriction to any 
compact subgroup $U\subseteq G$ of a smooth semilinear $G$-action $\Theta:G
\to\GL_{K^G}(V)$ on $V$ is given by the action $id_{\Lambda_{U,\Theta}}
\otimes\psi$ on $\Lambda_{U,\Theta}\otimes_{K^U}K=V$ for a $K^U$-lattice 
$\Lambda_{U,\Theta}$ in $V$. Fix a system $B$ of compact 
subgroups of $G$ such that (i) $B$ covers a dense subgroup in $G$, 
(ii) any pair of subgroups in $B$ is contained in a subgroup in $B$ 
(e.g., as $B$ we can take the collection of {\sl all} compact subgroups). 

Each smooth semilinear action of $G$ on $V$ determines a 
compatible system of $K^U$-lattices $\Lambda_U$ in $V$ for all 
$U\in B$ (in other words, an element of the set $\prlim_{U\in B}
\{\text{$K^U$-lattices $\Lambda_U$ in $V$}\}$): if 
$U\subseteq U'$ then $\Lambda_U=\Lambda_{U'}\otimes_{K^{U'}}K^U$. 

Suppose that $G$ is locally compact. Then we may assume that $B$ 
consists of {\sl open} compact subgroups of $G$ and there is 
a bijection between (a) the smooth semilinear actions of $G$ 
on $V$ and (b) compatible systems of $K^U$-lattices $\Lambda_U$ 
in $V$ for all $U\in B$. 

If we fix a $K^G$-lattice $\Lambda$ then the mapping $[g]\mapsto 
g(\Lambda\otimes_{K^G}K^U)$ identifies the set of $K^U$-lattices 
in $V$ with the set $\GL_K(V)/\GL_{K^U}(\Lambda\otimes_{K^G}K^U)$, 
and therefore, the set of smooth semilinear $G$-action on $V$ 
can be identified with the set  
$\prlim_{U\in B}\GL_K(V)/\GL_{K^U}(\Lambda\otimes_{K^G}K^U)$, 
and the set of isomorphism classes of smooth semilinear actions 
of $G$ on $V$ coincides with $\GL_K(V)\backslash
[\prlim_{U\in B}\GL_K(V)/\GL_{K^U}(\Lambda\otimes_{K^G}K^U)]$. 

If $G$ is not locally compact then the {\sl smooth} semilinear 
$G$-actions on a given $K$-vector space $V$ are described as 
compatible systems $(\Lambda_U)_{U\in B}$ of $K^U$-lattices in $V$ 
such that for any vector $v\in V$ the intersection over all $U\in B$ 
of the subfields generated over $K^U$ by the coordinates of $v$ with 
respect to the lattice $\Lambda_U$ is of finite type over $K^G$. 

{\sc Example.} If $G$ is countable at infinity then as $B$ one can choose 
a totally ordered collection $U_1\subseteq U_2\subseteq U_3\subseteq\dots$ 
of compact subgroups in $G$ such that $G=\bigcup_{m\ge 1}U_m$. Any 
compatible systems of $K^{U_i}$-lattices $\Lambda_{U_i}$ can be presented 
(a priori, not uniquely) as a composition, infinite to the left, 
$\cdots b_3b_2b_1b(\Lambda)$, where $b_i\in\GL_{K^{U_i}}
(\Lambda\otimes_{K^G}K^{U_i})$ (so that $\Lambda_i:=b_{i-1}\cdots 
b_2b_1b(\Lambda)\otimes_{K^G}K^{U_i}$). 

\vspace{5mm}

\noindent
{\sl Acknowledgements.} {\small I am grateful to the referee for 
suggesting numerous improvements of exposition. The project originates 
from my stay at the Max-Planck-Institut in Bonn, some 
of its elements were realized at the Institute for Advanced Study in 
Princeton, and then attained a form close the final one at the I.H.E.S. 
in Bures-sur-Yvette. I am grateful to these institutions 
for their hospitality and exceptional working conditions. 

This research was carried out within the National Research 
University Higher School of Economics Academic Fund Programme for 
2013--2014, research grant no. 12-01-0187. 
The author was member of the Institute for Advanced Study 
in Princeton, supported by the NSF grant DMS-0635607. The author 
also gratefully acknowledges the support by the Max-Planck-Institut 
f\"{u}r Mathematik in Bonn on the early stage of the project. 

}

\end{document}